\documentclass[a4paper,english]{elsarticle}
\usepackage[english]{babel}
\usepackage[utf8]{inputenc}
\usepackage{amsmath,amsthm,amssymb,graphicx, xcolor, xspace,complexity,subfig}

  \usepackage{tikz}
  \usetikzlibrary{positioning,fit,calc,shapes,backgrounds,matrix}

\theoremstyle{plain}
\newtheorem{theorem}{Theorem}[section]
\newtheorem{lemma}[theorem]{Lemma}
\newtheorem{corollary}[theorem]{Corollary}
\newtheorem{proposition}[theorem]{Proposition}
\newtheorem{claim}{Claim}[theorem]

\theoremstyle{definition}
\newtheorem{remark}[theorem]{Remark}
\newtheorem{problem}[theorem]{Problem}
\newtheorem{question}[theorem]{Question}

\DeclareMathOperator{\hull}{Hull}
\DeclareMathOperator{\icc}{I_{cc}}
\newcommand{\iccp}[1]{\mathop{\mathrm{I_{cc}^{\text{$#1$}}}}}
\DeclareMathOperator{\ifc}{I_{fc}}
\DeclareMathOperator{\ip}{I_{P_3}}
\DeclareMathOperator{\ips}{I_{P_3}^*}
\DeclareMathOperator{\hn}{hn}
\DeclareMathOperator{\hncc}{hn_{cc}}
\DeclareMathOperator{\incc}{in_{cc}}
\DeclareMathOperator{\hnfc}{hn_{fc}}
\DeclareMathOperator{\hnp}{hn_{P_3}}
\DeclareMathOperator{\hnps}{hn_{P_3^*}}

\newcommand{\CNF}{\textsf{CNF}\xspace}
\newcommand{\Nat}{\mathbb{N}}
\newcommand{\per}{per}

\newcommand{\thm}[2]{\vspace{0.2cm}\noindent\textbf{Theorem~#1.} \emph{#2}\vspace{0.2cm}}

\newenvironment{claimproof}[1]{\par\noindent\underline{Proof:}\space#1}{\hfill $\blacksquare$}

\journal{}

\begin{document}

\begin{frontmatter}
\title{Cycle convexity and the tunnel number of links}

\author[ufcmat]{Julio~Araujo~\fnref{t2,t4}}
\author[ufccomp]{Victor~Campos~\fnref{t4}}
\author[ufcmat]{Darlan~Gir\~ao}
\author[cmuc]{Jo\~ao~Nogueira~\fnref{t5}}
\author[cmuc]{Ant\'onio~Salgueiro~\fnref{t5}}
\author[ufcmat]{Ana~Silva~\fnref{t2,t4}}
\address[ufcmat]{Department of Mathematics, Federal University of Cear\'a, Fortaleza, Brazil}
\address[cmuc]{CMUC, Department of Mathematics, University of Coimbra, Coimbra, Portugal}
\address[ufccomp]{Department of Computer Science, Federal University of Cear\'a, Fortaleza, Brazil}

\fntext[emails]{\emph{E-mail addresses:} \texttt{\{julio, anasilva\}@mat.ufc.br} (J. Araujo and A. Silva),
\texttt{\{nogueira, ams\}@mat.uc.pt} (J. Nogueira and A. Salgueiro) and
\texttt{
campos
@lia.ufc.br} (V. Campos
).}

\fntext[t2]{Partially supported by CNPq Universal 401519/2016-3 and Produtividade 304576/2017-4, 304478/2018-0 and 313642/2018-4.}
\fntext[t4]{Partially supported by FUNCAP/CNPq/PRONEM PNE-0112-00061.01.00/16.}
\fntext[t5]{Partially supported by the Centre for Mathematics of the University of Coimbra - UIDB/00324/2020, funded by the Portuguese Government through FCT/MCTES.}


\begin{abstract}
    In this work, we introduce a new graph convexity, that we call Cycle Convexity, motivated by related notions in Knot Theory.
    
    For a graph $G=(V,E)$, define the interval function in the Cycle Convexity as $\icc(S) = S\cup \{v\in V(G)\mid \text{there is a cycle }C\text{ in }G\text{ such that } V(C)\setminus S=\{v\}\}$, for every $S\subseteq V(G)$. We say that $S\subseteq V(G)$ is convex if $\icc(S)=S$. The convex hull of $S\subseteq V(G)$, denoted by $\hull(S)$, is the inclusion-wise minimal convex set $S'$ such that $S\subseteq S'$. A set $S\subseteq V(G)$ is called a hull set if $\hull(S)=V(G)$. The hull number of $G$ in the cycle convexity, denoted by $\hncc(G)$, is the cardinality of a smallest hull set of $G$.
    
    We first present the motivation for introducing such convexity and the study of its related hull number. Then, we prove that: the hull number of a 4-regular planar graph is at most half of its vertices; computing the hull number of a planar graph is an $\NP$-complete problem; computing the hull humber of chordal graphs, $P_4$-sparse graphs and grids can be done in polynomial time. 
\end{abstract}

\end{frontmatter}


\section{Introduction}
\label{sec:intro}

In this work, we link two distinct areas of knowledge: Graph Convexity and Knot Theory. As it might interest researchers from different domains, we tried to present most of the needed notions in this section, but if the reader finds any gap, we refer to~\cite{BoMu08,GJ79,MT.book, BZ05}. 

In the sequel, we first shortly describe the notions we need from both areas, alongside the related works in the literature, before stating our results.

\subsection{Basics in Graph Theory}
\label{subsec:graphtheory}

A \emph{graph} $G$ is a pair $(V,E)$, where $V$ is any finite set called the \emph{vertex set of $G$}, and $E$ is a multi set of subsets of $V$ of size at most~2; it is called the \emph{edge set of $G$}. When $V$ and $E$ are not given, they are denoted by $V(G)$ and $E(G)$, respectively. The \emph{order} (resp. \emph{size}) of $G$ is $n(G)=|V(G)|$ (resp. $m(G)=|E(G)|$). If $\{u,v\}\in E$, we write simply $uv$. We say that $uv$ has \emph{multiplicity $k$} if $uv$ occurs $k$ times in $E$, and we write $\mu(uv) = k$. If $e=\{u\}\in E$, we call $e$ a \emph{loop}, and we say that $G$ is \emph{simple} if $G$ has no loops and every edge of $G$ has multiplicity one.

Here, we consider only loopless graphs. We shall see that this constraint is needed from the point of view of Graph Convexity and is a irrelevant constraint to our motivation in Knot Theory. 

Two graphs $H$ and $G$ are \emph{isomorphic} if there exist bijective functions $f:V(H)\to V(G)$ and $g:E(H)\to E(G)$ such that $e=uv\in E(H)$ if, and only if, $g(e)=f(u)f(v)\in E(G)$. It is well-known that isomorphism defines an equivalence relation in any family of graphs and that graphs in a same equivalence class have, intuitively, the same structure. All results here refer to \emph{unlabelled graphs}, i.e. the results hold to all graphs in a same equivalence class of the isomorphism relation (i.e. we just care how elements in the vertex set are joined by edges, but not what they are).
 
Given a graph $G$, the \emph{neighborhood of a vertex $u\in V(G)$} is the set $N(u) = \{v\in V(G)\mid uv\in E(G)\}$;  the \emph{neighborhood of a subset $X\subseteq V(G)$} is the set $N(X) = \bigcup_{v\in X}N(v)\setminus X$; the \emph{neighborhood of $u$ in $X\subseteq V(G)$} is the set $N_X(u) = N(u)\cap X$; and the \emph{neighborhood of a subset $X\subseteq V(G)$ in $X'\subseteq V(G)$} is the set $N_{X'}(X) = N(X)\cap X'$. The \emph{degree of $u$ in $G$} is denoted by $d_G(u)$ and equals $\sum_{v\in N(u)}\mu(uv)$ (we omit the subscript when $G$ is clear in the context). The \emph{minimum degree of $G$} is the minimum value over $d(u)$, $u\in V(G)$, and is denoted by $\delta(G)$. We say that $G$ is \emph{$k$-regular} if $d(u)=k$ for every $k\in V(G)$; this is also sometimes called \emph{$k$-valent}.

Given a subset $C\subseteq V(G)$, the \emph{subgraph of $G$ induced by $C$} is the graph $G[C]=(C,E_C)$, where $uv\in E_C$ if and only if $\{u,v\}\subseteq C$ and $uv\in E(G)$. If $S\subseteq V(G)$, then we define $G\setminus S = G[V(G)\setminus S]$. A graph $H$ is an \emph{induced subgraph} of $G$ if there is a set $S\subseteq V(G)$ such that $H$ is isomorphic to $G[S]$.

A $v_1,v_q$-\emph{walk} in $G$ (between $v_1$ and $v_q$) of \emph{length} $q-1$ is a sequence of vertices $(v_1,\cdots,v_q)$ such that either $q=1$ or $v_iv_{i+1}\in E(G)$, for every $i\in \{1,\cdots,q-1\}$ 
If $(v_1,\cdots,v_q)$ is a $v_1,v_q$-walk with no vertex repetitions, then it is a \emph{$v_1,v_q$-path}. If $(v_1,\cdots,v_q,v_1)$ is a closed walk with at least one edge such that the only vertex repetition in the sequence is $v_1$, then it is a \emph{cycle} of length $q$. A graph is said to be \emph{connected} is there exists a path between $u$ and $v$, for every pair of vertices $u,v\in V(G)$. Also, a \emph{component} of $G$ is a maximal connected (induced) subgraph of $G$. 

Given a graph $G$, a subset $S\subseteq V(G)$ is called a \emph{stable set} if there are no edges between two elements of $S$; and it is called a \emph{clique} if every pair of elements in $S$ are adjacent. The simple graph on $n$ vertices such that any two vertices are adjacent is the \emph{complete} graph, denoted by $K_n$.

A vertex $u\in V(G)$ is a \emph{cut-vertex} of $G$ if $G-u$ has more components than $G$. A graph $G$ is \emph{2-connected} if $G$ is connected and does not have any cut-vertex. A \emph{block of $G$} is a maximal 2-connected subgraph of $G$.  It is well-known that the block structure of a connected graph is that of a tree. In fact, the \emph{block-cutpoint tree} of a simple connected graph $G$ is the tree that has one vertex for each cut-vertex of $G$ and one vertex for each block of $G$, and two vertices $x$ and $y$ are joined by an edge in this tree if, and only if, $x$ represents a cut-vertex of $G$ that belongs to the vertex set of the block represented by vertex $y$. Thus, note that the leaves of this tree correspond to blocks of $G$ having exactly one cut-vertex. Such blocks are called \emph{leaf-blocks}.

The embedding of a graph $G$ in the Euclidean plane is a function assigning to distinct vertices in $V(G)$ distinct points in the plane, and assigning to each edge $uv\in E(G)$ a (polygonal simple) curve from the point associated to $u$ to the point associated to $v$.
A graph $G$ is \emph{planar} if $G$ can be embedded on the plane in a way that its edges intersect only at their endpoints, i.e. no two curves associate to the edges cross. The embedding of a planar graph without crossings is called a \emph{plane graph}, i.e. a graph whose vertices are points and edges are curves that do not cross. Given a plane graph $G$, each region of the plane minus $G$ is called a \emph{face of $G$}, and the unbounded face is called the \emph{outerface of $G$} (recall we consider $V(G)$ to be finite). The set of faces of $G$ is denoted by $F(G)$. The boundary of a face $f\in F(G)$ is the set of edges of $G$ bounding the region corresponding to $F$. We say that a face $f\in F(G)$ has \emph{degree $k$} if the length of the shortest closed walk in its boundary equals $k$. A face of degree $k$ is also called a \emph{$k$-face}.

The \emph{contraction} of an edge $e=uv\in E(G)$ of a graph $G$ results in a graph $G'$, denoted by $G'= G/e$, such that $G'$ is obtained by the removal of the edge $e$ and the merge of the vertices $u$ and $v$ into a single new vertex $w$ being incident to the same edges as $u$ and $v$ in $G$. A graph $H$ is a \emph{minor} of $G$ if $H$ is obtained from $G$ by edge or vertex deletions and edge contractions. It is well-known that planar graphs is a \emph{minor-closed} graph class, i.e. any minor of a planar graph is also planar.

A \emph{forest} is an acyclic graph. A \emph{tree} is a connected forest. A \emph{chordal graph} is a graph with no induced cycles of length greater than~3. A \emph{$P_4$-sparse graph} is a graph in which every subset of 5 vertices induces at most one path on~4 vertices. For positive integers $m$ and $n$, the \emph{$m\times n$ grid}, denoted by $G_{m,n}$ is the graph on vertices $\{(i,j)\mid 1\le i\le m, 1\le j\le n\}$ having as edge set:
\[\begin{array}{l}\{(i,j)(i,j+1)\mid i\in\{1,\ldots,m\}, j\in\{1,\ldots,n-1\}\}\cup\\
 \{(i,j)(i+1,j)\mid i\in\{1,\ldots,m-1\}, j\in\{1,\ldots,n\}\}.\end{array}\]

\subsection{Graph Convexities}
\label{subsec:graphconvexities}

A \emph{convexity space} is an ordered pair $(V,\mathcal{C})$, where $V$ is an arbitrary set and $\mathcal{C}$ is a family of subsets of $V$, called \emph{convex sets}, that satisfies:
\begin{itemize}
  \item[{\bf (C1)}] $\emptyset,V\in \mathcal{C}$; and
  \item[{\bf (C2)}] For all $\mathcal{C'}\subseteq \mathcal{C}$, we have $\bigcap \mathcal{C'}\in \mathcal{C}$.
  \item[{\bf (C3)}] Every nested union of convex sets is convex.
\end{itemize}
Given a subset $S\subseteq V$, the \emph{convex hull of $S$} (with respect to $(V,\mathcal{C})$) is the (inclusion-wise) minimum $C\in \mathcal{C}$ containing $S$ and we denoted it by $\hull(S)$. If $\hull(S) = V$, then we say that $S$ is a \emph{hull set}. 

In Graph Convexity, convexity spaces are defined over the vertex set of a given graph $G$. Thus, we always consider the set $V$ to be finite and non-empty. We mention that this makes Condition (C3) irrelevant, since any such union will have a larger set. 

Given a graph $G = (V,E)$ and a function $I: 2^V\to 2^V$, we say that a subset $S\subseteq V$ is \emph{$I$-closed} if $I(S) = S$. Moreover, for a subset $S\subseteq V$ let:
$$I^k(S) = \begin{cases}I(S), \text{ if }k=1;\\I(I^{k-1}(S)),\text{ otherwise.}\end{cases}$$

Define 
\[\icc(C) = C\cup \{v\in V(G)\mid \mbox{there exists a cycle in $G[C\cup\{v\}]$ containing $v$}\}.\]

\begin{proposition}
\label{prop:cycleConvexity}
 Let $G = (V,E)$ be a graph. Let $\mathcal{C}$ be the family of $\icc$-closed subsets of vertices of $G$. Then, $(V,\mathcal{C})$ is a convexity space.
\end{proposition}
\begin{proof}
Since $G$ is loopless, note that $\icc(\emptyset) = \emptyset$. By definition, it is clear that $\icc(V) = V$. Since $V$ is finite, it suffices then to only check condition (C2).

By contradiction, let $\mathcal{C'}\subseteq \mathcal{C}$ be such that $\bigcap\mathcal{C'}\notin \mathcal{C}$; denote $\bigcap\mathcal{C'}$ by $C$. From the definition of $\mathcal{C}$ and of $\icc$-closed sets, we get that $\icc(C)\neq C$. Because $C\subseteq\icc(C)$, it means that there exists $w\in \icc(C)\setminus C$, and thus there is a cycle $X$ in $G[C\cup \{w\}]$ containing $w$. Note that all vertices of $X$ except $w$ belong to $C$.

Now, write $\mathcal{C'}$ as $\{C_1,\ldots, C_p\}$, for some positive integer $p$ (recall that $\mathcal{C}$ is finite), and suppose, without loss of generality, that $w\notin C_1$. Since all vertices of $X$, but $w$, belong to $C_1$, we have that $w\in\icc(C_1)$. This contradicts the hypothesis that $C_1\in \mathcal{C'}\subseteq \mathcal{C}$, as $C_1\neq \icc(C_1)$.
\end{proof}
We emphasize that the loopless constraint in $G$ is important to ensure that the empty set is $\icc$-closed. In view of Proposition~\ref{prop:cycleConvexity}, we say that the family $\mathcal{C}$ of $\icc$-closed subsets of a given loopless graph $G$ is the \emph{cycle convexity} over $V(G)$ and that $(V,\mathcal{C})$ is the \emph{cycle convexity space} over the graph $G$. We often refer to an $\icc$-closed subset $S$ as \emph{convex} in the cycle convexity and we say that $V(G)\setminus S$ is \emph{co-convex}.

As already mentioned before in a more general context, a subset $S$ is a \emph{hull set} in the cycle convexity if $\hull(S) = V(G)$. The \emph{hull number} in the cycle convexity, denoted by $\hncc(G)$, is the cardinality of a smallest hull set of a graph $G$ in the cycle convexity. This is the parameter we study in this work.

One should notice that, given a graph $G = (V,E)$ and a subset $S\subseteq V(G)$, we have that $\hull(S) = \iccp{k}(S)$, for a sufficiently large $k$ (such as $k=n(G)$, for instance). Consequently, $\hull(S)$ can be obtained by iteratively adding vertices in $\iccp{k}(G)\setminus\iccp{k-1}(G)$, for $k\in\{1,\ldots, n(G)\}$. Then, we say that a vertex $v$ of $G$ is \emph{contaminated} (some authors also use \emph{percolated} or \emph{generated}) at step $k=0$ if $v\in S$, otherwise at step $k\ge 1$ where $k$ is the minimum positive integer such that $v\in\iccp{k}(S)\setminus\iccp{k-1}(S)$. In case $S$ is not a hull set of $G$, for some vertices that do not belong to $S$ such minimum $k$ will not exist and we say that they are not contaminated by $S$.

In order to relate our results to Knot Theory, let us also define a closely related parameter.
For a plane graph $G$, let $V(F)$ denote the vertices that are incident to the face $F$ of $G$. We also define the \emph{face convexity} over a plane graph $G$ by using the function $\ifc:2^{V(G)}\rightarrow 2^{V(G)}$, defined as follows:
$$\ifc(C) = C\cup \{v\in V(G)\mid \text{there is a face in } G[C\cup\{v\}]\text{ containing } v\}.$$

One can similarly define the $\ifc$-closed closed subsets of vertices of a given plane graph $G$ and prove that the family of $\ifc$-closed subsets define a convexity over $V(G)$ - such convexity we call \emph{face convexity}. We also similarly define the \emph{hull number} in the face convexity and denote it by $\hnfc(G)$. By definition, note that $\hncc(G)\leq\hnfc(G)$, for any plane graph $G$.

For shortness we refer to  $\hncc(G)$ simply as the \emph{hull number of $G$} and to $\hnfc(G)$ as the \emph{face hull number of $G$}.  Similarly, by hull set, we mean the hull set on the cycle convexity, and by \emph{face hull set}, a hull set in the face convexity.
 

As we shall see in the next section, the definition of such functions leading to graph convexities is motivated by an invariant of knots and links called the tunnel number. Before going into these details in Knot Theory, let us first refer to previous works in the literature concerning Graph Convexity.

It is important to emphasize that there are several similar functions defined over the vertices of a graph in the literature, leading to several graph convexities, e.g.: geodetic~\cite{FJ.86}, $P_3$~\cite{PWW.08}, $P_3^*$~\cite{ASSS.18}, monophonic~\cite{D.88}. 
The graph convexities we should emphasize are the $P_3$ and $P_3^*$. Their related functions are, respectively: $\ip(C) = C\cup \{u\in V(G)\mid u \text{ has two neighbors in $C$}\}$ and $\ips(C) = C\cup \{u\in V(G)\mid u \text{ has two \emph{non-adjacent} neighbors in $C$}\}$. The related similar hull numbers are also defined and denoted by $\hnp(G)$ and $\hnps(G)$.

We should mention that the $P_3$-convexity is also closely related to the notion of \emph{percolation}~\cite{BR06}; this concept has many applications in physics and they are usually interested in probabilistic results, e.g., the probability $p$ such that if the vertices are initially chosen with probability $p$, then the chosen set is a hull set in the $P_3$-convexity, i.e. it percolates the entire graph.

Observe that $\hnp(G)=\hnps(G)$ if $G$ has no triangles. Also, because a hull set in the $P_3^*$-convexity and in the cycle convexity is also a hull set in the $P_3$-convexity, we get that $\hnp(G)\le \hnps(G)$ and $\hnp(G)\le \hncc(G)$. Moreover, unless $G$ is triangle-free, one cannot ensure a relationship between $\hnps(G)$ and $\hncc(G)$. In Subsection~\ref{sec:ourResults}, when presenting our results, we will comment about some known results on these convexities.

\subsection{Knot Theory}

In the sequel, we present basic notions of Knot Theory and describe a problem in this area which is related to the hull number (with respect to the interval functions of the previous section) of planar graphs arising from knots and link diagrams. Throughout this paper we work in the smooth category of manifolds. 

\subsubsection*{Basics in knot theory}

A \textit{knot} is an embedding of a circle in the 3-dimensional sphere $S^3$. If $f:S^1\rightarrow S^3$ is such an embedding, we think of the knot as the subset $K\subset S^3$ given by $K=f(S^1)$. A finite collection of disjoint knots in $S^3$ is called a  \textit{link}. Hence, knots are single component links.  Two knots or links $K_1, K_2$ are said to be \textit{equivalent} if they are \textit{ambient isotopic}.
An ambient isotopy $\phi:S^3\times [0,1]\rightarrow S^3$ is a continous map such that  each $\phi_t=\phi|_{S^3\times\{t\}}:S^3\equiv S^3\times\{t\}\rightarrow S^3$ is a homeomorphism and $\phi_0=\text{id}_{S^3}$. To  say $K_1$ and $K_2$ are equivalent means that such a $\phi$ exists, satisfying $\phi_1(K_1)=K_2$.  

A \textit{polygonal knot} is a knot whose image is the union of a finite set of line segments.  We shall only consider knots which are equivalent to  polygonal knots. These are called \textit{tame knots}.  

Let  $K\subset S^3$ be a knot or link. A useful way to visualize $K$ is to consider its projection on a plane.   It is a well known fact that, generically, such a projection is one-to-one, except at a finite number of double points where the projection crosses itself transversely. These projections correspond to finite graphs where all vertices are 4-regular. From these graphs one builds  \textit{diagrams} for $K$ as  follows: at each vertex we distinguish between the over-strand and the under-strand of $K$ by creating a break in the strand going underneath. See Figure \ref{diagram}.    
Such a diagram is denoted by $D(K)$ and the corresponding 4-regular graph is denoted by $G_{D(K)}$. 

\begin{figure}[!ht]
	\centering
	\includegraphics[scale=0.15]{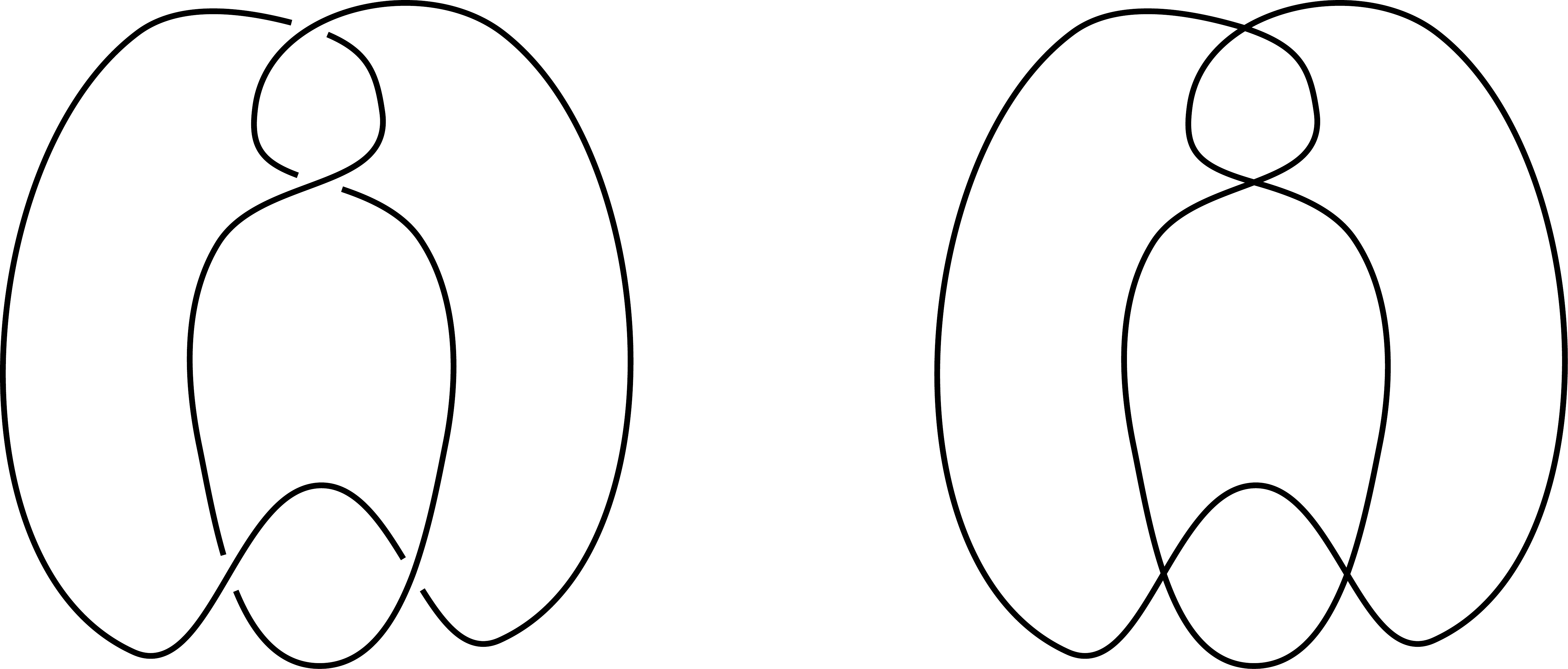}
	\caption{Left: a diagram of a knot; Right: the corresponding 4-regular graph.}
	\label{diagram}
\end{figure}

We remark that the correspondence between the graphs and diagrams is not one-to-one. The same graph may be associated to several distinct diagrams. For example, it can be easily proved that any finite connected 4-regular graph is associated to a diagram of the trivial knot (the knot which bounds a disk embedded in $S^3$).    

\subsubsection*{Tunnel number}

Given a knot or link $K$ in $S^3$, a \textit{tunnel} for $K$ is a properly embedded arc in the exterior $E(K)$ of $K$. Given a collection of disjoint tunnels $\mathcal{T}=\{t_1,...,t_m\}$ for $K$, we may think of $K\cup\mathcal{T}$ as a spatial graph. The collection $\mathcal{T}$ is called an \textit{unknotting tunnel system} for $K$ if a regular neighborhood of the graph $K\cup\mathcal{T}$ can be isotoped into a regular neighborhood of a plane graph in $S^3$. Every link $K$ in $S^3$ admits an unknotting tunnel system. In fact, if we add one vertical tunnel at each crossing of a diagram $D(K)$, we obtain an unknotting tunnel system for $K$. The \textit{tunnel number} of $K$ is the minimal cardinality among all unknotting tunnel systems for $K$ and we denote it by $t(K)$.

The tunnel number is also interpreted in the setting of Heegaard genus $g(M)$ of a 3-manifold $M$, which is the minimal genus of a surface splitting the 3-manifold into two compression bodies. (See for instance \cite{M07}.) Note that the definition of an unknotting tunnel system $\mathcal{T}=\{t_1,...,t_m\}$ is equivalent to say that the boundary of a regular neighborhood of $K\cup\mathcal{T}$, which is a surface of genus $m+1$, splits $E(K)$ into a handlebody and a compression body. We then have that $g(E(K))=t(K)+1$. The computation of the Heegaard genus of a $3$-manifold is very difficult to compute in general. In fact, by \cite{BDS17}, this problem is NP-hard.


In some situations the tunnel number can be computed exactly. For instance, if one single tunnel defines an unknotting tunnel system for a non-trivial knot, then the knot has tunnel number one. Also, it is known that tunnel number one knots are prime. So, if a composite knot has an unknotting tunnel system with two tunnels, then the tunnel number of the knot is 2. But determining if prime knots have tunnel number 2 or higher is a difficult task. In general, one works with upper bounds or lower bounds for the tunnel number. For instance, as observed above, the crossing number is an upper bound for the tunnel number, and, from the Heegaard splitting of the knot exterior, the rank of the fundamental group of the knot exterior is a lower bound for the tunnel number plus one. 
For links the tunnels necessarily have to connect the link components. Hence, its tunnel number is at least the number of components minus one. This lower bound can be used to determine the tunnel number of a link in case one finds an unknotting tunnel system with that cardinality, as in \cite{Lichtenstein82}. Also in \cite{GNS17} the authors determine the tunnel number of a class of links by exploring the following upper bound: given a diagram $D(K)$ of $K$, consider vertical arcs added at certain crossings (see figure \ref{vertical_arc} below). What is the smallest unknotting tunnel system consisting of vertical arcs only? One of the  results in \cite{GNS17} shows that this is at most the hull number associated to the interval function $\ifc$ of the graph $G_{D(K)}$.  We conjecture that the tunnel number is also bounded above by $\hncc$. In fact, we developed a computer program to compute the hull number and the face hull number of every prime knot up to 12 crossings, whose results are described in Table \ref{table:knots}. For every one of these 2977 knot diagrams $D(K)$,  we obtained $\hncc(G_{D(K)})=\hnfc(G_{D(K)})$.
	
	 \begin{table}
		\begin{center}
			\begin{tabular}{|c|c|c|c|c|c|c|c|c|c|c|} 
				\hline
				{\bf Crossing number} & 3 & 4 & 5 & 6 & 7 & 8 & 9 & 10 & 11 & 12\\\hline
				$\hncc(G)=\hnfc(G)=1$ & 1 & 1 & 2 &3 &7&12&24&45&91&176\\\hline
				$\hncc(G)=\hnfc(G)=2$ & 0 & 0 & 0 &0&0&9&25&120&446&1952\\\hline
				$\hncc(G)=\hnfc(G)=3$ & 0 & 0 & 0 & 0&0&0&0&0&15&48\\\hline
				total & 1 & 1 & 2 & 3&7&21&49&165&552&2176\\\hline
			\end{tabular}
			\caption{The number of prime knots $K$ by crossing number and hull number of $G=G_{D(K)}$.}\label{table:knots}
		\end{center}
	\end{table}

Hence, finding  good upper bounds for the hull number of $G_{D(K)}$ would help us estimate $t(K)$.

\begin{figure}[!ht]
	\centering
	\includegraphics[scale=0.1]{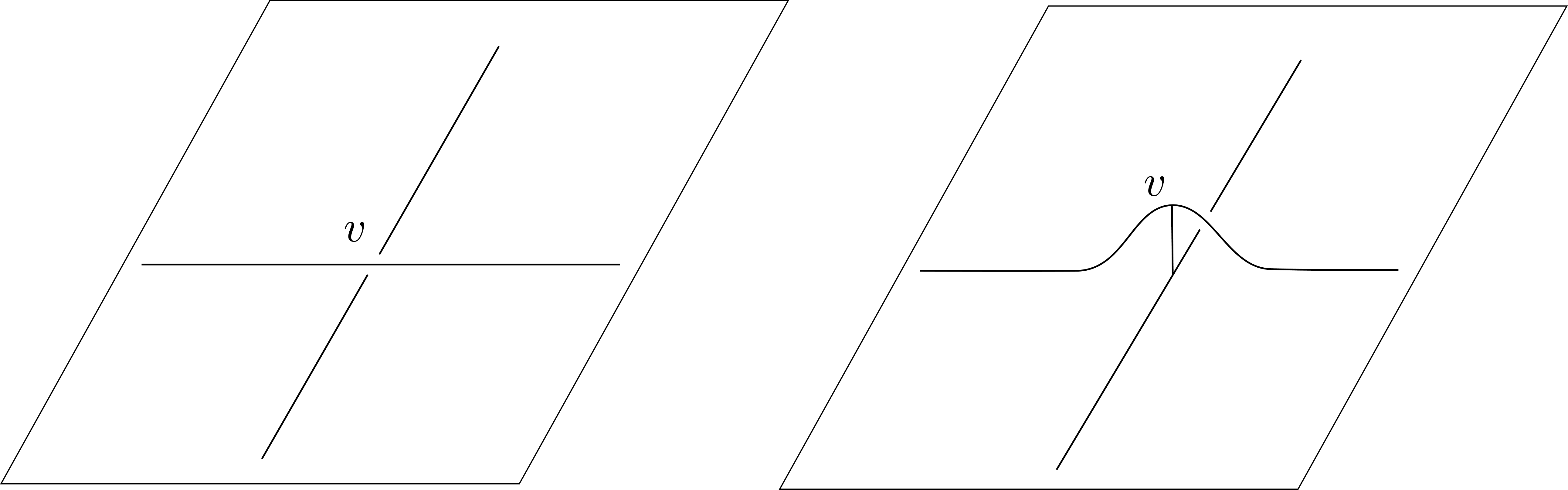}
	\caption{Adding a vertical arc at a crossing $v$ of the diagram.}
	\label{vertical_arc}
\end{figure}

\subsection{Relationship between hull number and tunnel number}

The result of this subsection have appeared, in a more general setting, in \cite{GNS17}. We include it here for completeness. 

\begin{remark}\label{relationship}   Suppose  one starts to add vertical arcs and, at some point,  there is a face  $f$ corresponding to one of the planar regions determined by the diagram $D(K)$ such that  all, except one, of the crossings of $f$ has a vertical arc.  Let $v$ represent this crossing. Then the crossing $v$ of $f$  may be removed. Let $K_1$ be the resulting 1-complex.  This is described in Figure \ref{figure:removal}.  
\begin{figure}[h]
	\centering
	\includegraphics[scale=0.13]{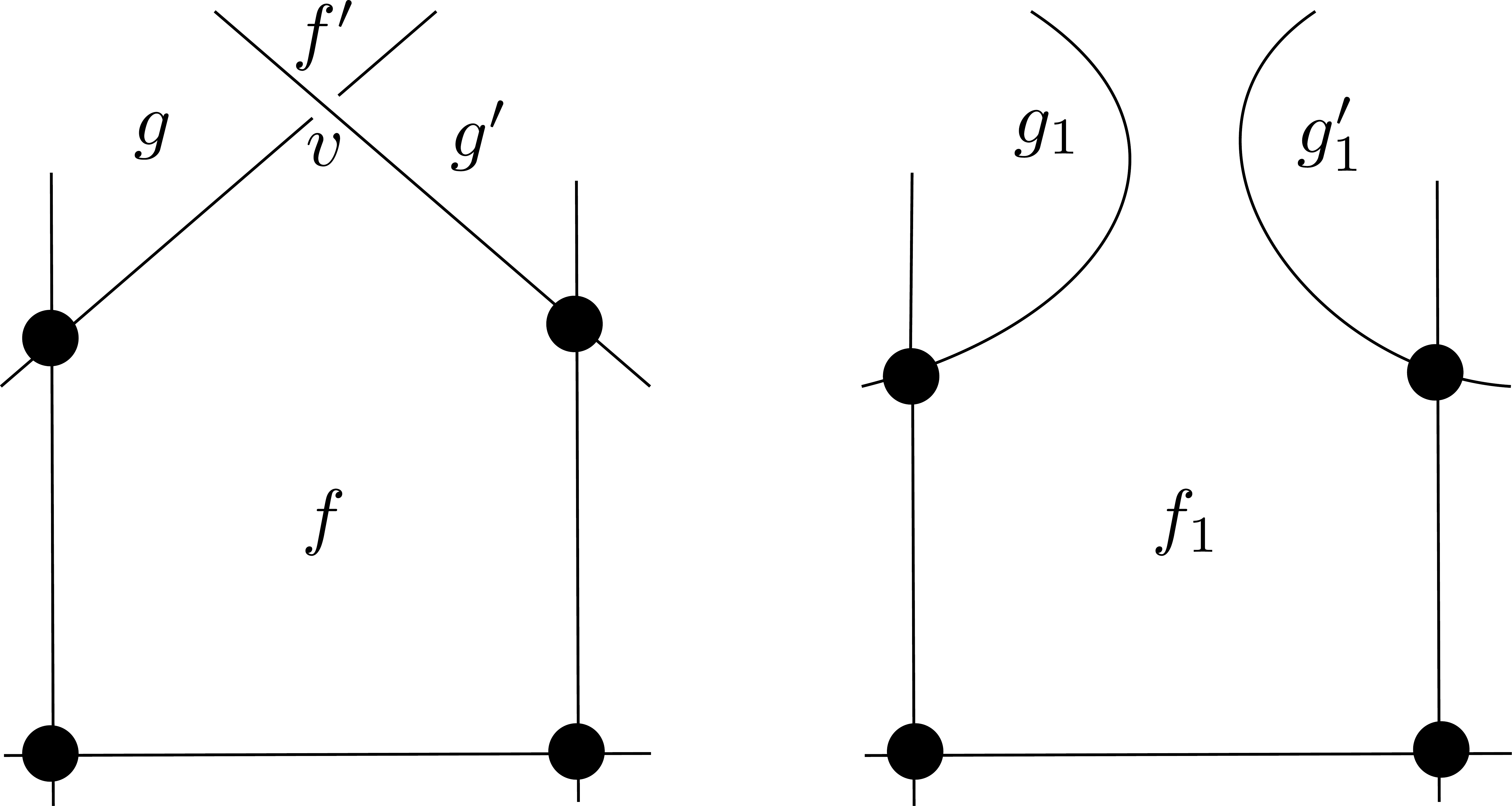}
	\caption{Removal of vertex}
	\label{figure:removal}
\end{figure}

In the diagram $D(K)$, let $f'$ be the face containing $v$, opposite to $f$. After $v$ is removed, as above, the faces $f$ and $f'$ merge into a single face $f_1$ of $D(K_1)$.
If necessary, we may keep adding vertical arcs. Whenever this procedure yields a face in which all crossings, except one, have a vertical arc, then the remaining crossing can be removed. 
\end{remark}

In what follows,  we restrict to minimal crossing diagrams of knots and links. Consider the 4-regular  graph $G_{D(K)}$ given by a diagram $D(K)$ of a link $K$, i.e., just ignore over/under information on $D(K)$.  

\begin{theorem}\label{upper bound}
$t(K)\leq\hnfc(G_{D(K)})$.
\end{theorem}

We thus wish to find good upper bounds for $t(K)$ in terms of the \textit{crossing number} of $K$. The proof of this theorem relies on   Remark \ref{relationship}.

Let $V_0$ be a subset of vertices in $D(K)$ in which vertical arcs have been added. Let $f$ be a face of $D(K)$ and suppose that vertical arcs have been added in all, except one, vertices of $f$.  In the link $K$, let $v$ represent the remaining crossing of $f$, as in the remark. Consider the 1-complex $K_1$ obtained by removing this crossing, according to the remark, and $D(K_1)$ the planar graph associated to $K_1$.  

\begin{lemma}  Percolation in the graph  $D(K_1)$ is equivalent to the percolation  in $D(K)$. \label{lem:tunnel}
\end{lemma}

\begin{proof} 
The vertex (crossing) $v$ is  removed from $D(K)$. 
Let $f_1$ be the face of $D(K_1)$  obtained  by merging the faces $f$ and $f'$.  Then the vertices of $f_1$ that do not have a vertical arc  are exactly the vertices of $f'$ that that do not have vertical arcs.  Moreover,  let $g, g'$ be the faces of $D(K)$ adjacent to $v$ other than $f,f'$. These yield two faces $g_1, g_1'$ in $D(K_1)$, respectively. As before, the vertices of $g_1, g_1'$ that do not have a vertical arc are exactly those vertices of $g, g'$ that do not have a vertical arc.  
\end{proof}

\begin{proof}[Proof of theorem \ref{upper bound}]

The theorem will be proved by showing $\hnfc(D(K))$ is an upper bound for the number of vertical arcs needed to make the resulting 1-complex planar. Assume $V_0$ is a hull set for $D(K)$, i.e.,  coloring all vertices in $V_0$, then it percolates to the whole graph $D(K)$.  We  relate  the percolation rule to  Remark \ref{relationship}. 

Place one vertical arc at each crossing corresponding to the vertices of $V_0$. Let $v$ be the first crossing to be removed in $D(K)$.  Then it belongs to a face $f$ in which all other vertices were previously colored. By Remark \ref{relationship}, the crossing corresponding to $v$ may be removed. Let $D(K_1)$ be the graph obtained by removing the crossing $v$.   By Lemma~\ref{lem:tunnel}, percolation in $D(K)$ is equivalent to percolation in $D(K_1)$. Then $V_0$ is a hull set for $D(K_1)$ also. Inductively, repeat  this process to the remaining  vertices (crossings). We end up with a diagram $D'$ of  1-complex  $K'$ in which there are no crossings, i.e.,   $K'$ is planar.

\end{proof}

\subsection{Our Results}\label{sec:ourResults}

The main initial motivation for this work has been to find good upper bounds for the hull number of 4-regular planar graphs, which are loopless but not necessarily simple, in order to obtain better bounds for the tunnel number of links. At the beginning of our studies, we believed that a version of Theorem~\ref{upper bound} would also hold for $\hncc(G)$. However, we have not been able to prove such a theorem and leave the following open question, whose positive answer would imply the desired result:

\begin{question}
Does $\hnfc(G)\le \hncc(G)$ hold, for a 4-regular planar graph $G$?
\end{question}




In Section~\ref{sec:upperbound}, we show that the $\hncc(G)\leq \frac{n(G)}{2}$, whenever $G$ is a 4-regular graph. Moreover, this is a tight bound.

It is worth mentioning that we have not been able to find a \emph{simple} 4-regular planar graph attaining such bound. In fact, we believe that this could be improved to a third if $G$ is a simple graph; hence we ask:

\begin{question}
What is the minimum $c$ such that $\hncc(G)\le c\cdot n$, for all simple 4-regular planar graphs $G$ on $n$ vertices?
\end{question}

It is also worth mentioning that, in general, planar graphs may need as many as all the vertices to be contaminated. In fact, in Section~\ref{sec:preliminaries}, besides prooving some properties about co-convex sets that we use in several results, we show that $\hncc(G)=n(G)$ if, and only if $G$ is a forest. We also show that if $G$ is a chordal graph with $p$ blocks, then $\hncc(G) = p+1$, and that $\hncc(G_{m\times n}) = m+n-1$.

Apart from the topological interest, the new graph convexity has an interest of its own. We recall the connection of our graph convexity to other previously known ones: the $P_3$ and $P_3^*$ convexities.
Concerning the $P_3$-convexity, in~\cite{CDPRS.11} the authors prove that computing the hull number is $\NP$-complete in general, and they give polynomial results for chordal graphs and cographs. In~\cite{ASSS.18}, the authors prove that computing the hull number in the $P_3$-convexity remains $\NP$-hard even when restricted to subgraphs of the grid. It implies that it is $\NP$-hard for bipartite planar graphs with bounded maximum degree. They also give a polynomial-time algorithm to compute $\hnp(G)$, for any $P_4$-sparse graph $G$. In addition, they introduce and investigate the $P_3^*$-convexity. Since in any triangle-free graph $G$ we have that $\hnp(G)=\hnps(G)$, they deduce from the previous result that computing $\hnps(G)$ is also $\NP$-hard whenever $G$ is a planar bipartite bounded-degree graph. They also present a polynomial-time algorithm to compute $\hnps(G)$, for any $P_4$-sparse graph $G$. Furthermore, in~\cite{DPR2019}, the authors show the $\NP$-hardness of computing $\hnps(G)$, for chordal graphs $G$.

A natural question is whether any of these results also apply to our convexity. In~Section~\ref{sec:npcomplete}, we prove that computing $\hncc(G)$ is $\NP$-complete even when $G$ is a planar graph, and in Section~\ref{sec:poly} we give a polynomial-time algorithm to compute the hull number in the cycle convexity of $P_4$-sparse graphs. Note that the previously mentioned equations that present in Section~\ref{sec:preliminaries} imply linear-time algorithms to compute the hull number of chordal graphs and grids. Table~\ref{table:results} summarizes the known complexity results about the computation of hull number in the graph classes investigated so far in the cycle convexity. 
 
\begin{table}
\begin{center}
\begin{tabular}{|c|c|c|c|} 
\hline
{\bf Graph class} & $\mathbf{P_3}$ & $\mathbf{P^*_3}$ & {\bf cycle conv.} \\\hline
chordal & $\P$ & \NP-complete & $\P$\\\hline
$P_4$-sparse & $\P$ & $\P$ & $\P$\\\hline
bipartite & $\NP$-complete & $\NP$-complete & ? \\\hline
planar & $\NP$-complete & $\NP$-complete & \NP-complete\\\hline
\end{tabular}
\caption{The complexity of computing the hull number of $G$ in the given convexities. In the table, $\P$ stands for polynomial, $\NP$ stands for $\NP$-complete and ``?'' stands for open.}\label{table:results}
\end{center}
\end{table}

In particular, our $\NP$-completeness proof does not bound the maximum degree, which means that we do not know what is the complexity of computing $\hncc(G)$ when $G$ is a 4-regular planar graph, our main class of interest. We therefore pose the following question:

 \begin{question}
Let $G$ be a 4-regular planar graph. Can one compute $\hncc(G)$ in polynomial time?
 \end{question}

We mention that there is a work that has been prepared after this one, but published first, dealing with the parameter \emph{interval number} in the Cycle Convexity~\cite{ADNS2018}. An \emph{interval set} $S\subseteq V(G)$ is any subset that satisfies $\icc(S)=V(G)$. The \emph{interval number} of a graph $G$, denoted by $\incc(G)$, is the minimum size of an interval set of $G$. They present several results on this parameter, giving bounds and as well as computational complexity results.



\section{Preliminaries}\label{sec:preliminaries}

In this section, we present some general simple results about the hull number in the cycle convexity. We also use such results to determine the exact value of $\hncc(G)$, whenever $G$ is a simple chordal graph or a grid.

When dealing with the parameter hull number (in any graph convexity), the study of co-convex sets is extremely important as they determine the value of the hull number. The proposition below follows from the definitions.

\begin{proposition}
\label{prop:co-convex}
Let $G$ be a graph and let $S,C\subseteq V(G)$ be such that $C\cap S=\emptyset$. If $S$ is a co-convex set, then $C$ is not a hull set of $G$.
\end{proposition}

Consequently, a hull set can be seen as a \emph{hitting-set} (or \emph{transversal}) of the family of co-convex sets of $G$.
Given a graph $G = (V,E)$, by definition, one should notice that a subset $S\subseteq V$ is co-convex in the cycle convexity if each vertex $v\in S$ does not belong to a cycle $C$ of $G$ such that $v$ is the unique vertex of $V(C)\cap S$. Equivalently, $S\subseteq V$ is co-convex in the cycle convexity if each vertex $v\in S$ has at most one edge linking $v$ to some vertex in a connected component of $G[V\setminus S]$.
A particular case of co-convex sets is the one of singletons.

\begin{proposition}
\label{prop:coconvex_singleton}
Given a graph $G = (V,E)$ and $v\in V$, then $\{v\}$ is a co-convex set in the cycle convexity if, and only if, $v$ lies in no cycle of $G$.
\end{proposition}
\begin{proof}
Suppose first, by contradiction, that $\{v\}$ is co-convex and that $G$ has a cycle $C$ containing $v$. By hypothesis, $V\setminus\{v\}$ should be convex, i.e. $\icc(V\setminus\{v\}) = V\setminus\{v\}$. However, due to the cycle $C$, $v\in \icc(V\setminus\{v\})$, a contradiction.

On the other hand, suppose that $v$ lies in no cycle of $G$. Consequently, $v$ cannot have two edges to a same connected component of $G-v$. Thus, $\{v\}$ is co-convex.
\end{proof}

\begin{proposition}\label{prop:nocriticalvertex}
Let $G$ be a graph and $w\in V(G)$. Then, every minimum hull set $S$ of $G$ must contain $w$ if, and only if, $\{w\}$ is a co-convex set of $G$.
\end{proposition}
\begin{proof}
We already argued that if $\{w\}$ is a co-convex set of $G$, then $w\in S$, for any (minimum) hull set $S$ of $G$. It just remains to prove the converse.

Let us prove, by the contrapositive, that if $\{w\}$ is not a co-convex set of $G$, then there exists a minimum hull set $S'$ of $G$ that does not contain $w$.
Thus, suppose that $\{w\}$ is not a co-convex set of $G$. If no minimum hull set of $G$ contains $w$, we have nothing to prove. Thus, let $S$ be a minimum hull set of $G$ such that $w\in S$. We shall construct a hull set $S'$ of $G$ such that $w\notin S'$ and $|S'|=|S|=\hncc(G)$.

Since $S$ is a minimum hull set of $G$, $\hull(S\setminus \{w\})$ does not contain all the vertices of $G$. In particular, we observe that $w\notin \hull(S\setminus \{w\})$, as otherwise $S\setminus\{w\}$ would be a smaller hull set of $G$. Moreover, $\hull(S\setminus \{w\})\neq V(G)\setminus\{w\}$ as $\{w\}$ is not co-convex and, again, $S\setminus\{w\}$ would be a smaller hull set of $G$. Consequently, there exist a vertex $w'\neq w$ such that $\{w,w'\}\cap \hull(S\setminus \{w\}) = \emptyset$.

Intuitively, at least one vertex $w'$ of $G$ depend on $w$ to be included in $\hull(S)$. Let us chose a first $w'$ to be included in $\hull(S)$ thanks to $w$. 
Formally, let $k$ be the smallest positive integer satisfying the following property: there is a vertex $w'\in V(G)$ such that $w'\in \iccp{k}(S)\setminus \iccp{k-1}(S)$, $w'$ is the unique vertex of a cycle $C$ that does not belong to $\iccp{k-1}(S)$ and $w\in V(C)$. By the previous remarks, note that such $k$ must exist.

Let $S' = (S\setminus\{w\}) \cup \{w'\}$. By the minimality of $k$, note that $\iccp{k-1}(S\setminus\{w\}) = \iccp{k-1}(S)\setminus \{w\}$. In particular, $w$ and $w'$ are the only vertices of $V(C)$ that do not belong to $\iccp{k-1}(S\setminus\{w\})$. Since we add $w'$ to $S'$, note that $w\in \iccp{k}(S')$ and thus $\hull(S)=\hull(S')$. This implies that $S'$ is a minimum hull set of $G$ that does not contain $w$.
\end{proof}

We shall use Propositions~\ref{prop:coconvex_singleton} and~\ref{prop:nocriticalvertex} to deduce our main result in Section~\ref{sec:upperbound}. Let us first use them to deduce simpler facts regarding co-convex sets with a single vertex.

\begin{corollary}
\label{cor:removesingletoncoconvex}
Given a graph $G = (V,E)$ and $v\in V$ such that $v$ lies in no cycle of $G$, then $\hncc(G) = \hncc(G-v)+1$.
\end{corollary}
\begin{proof}
Let $H=G-v$. First, let $S_H$ be a minimum hull set of $H$. Let $S = S_H\cup \{v\}$. Since $H$ is a subgraph of $G$ and $S_H$ is a hull set of $H$, note that $V(G)\setminus\{v\}\subseteq \hull(S)$. Consequently, $\hncc(G)\leq |S|=|S_H|+1=\hncc(G-v)+1$.

On the other hand, let $S_G$ be a minimum hull set of $G$. By Proposition~\ref{prop:coconvex_singleton}, note that $\{v\}$ is co-convex and thus $v\in S_G$. We claim that $S=S_G\setminus\{v\}$ is a hull set of $H$. In fact, since $v$ belongs to no cycle of $G$, no vertex is contaminated thanks to a cycle containing $v$. Consequently, $\hull(S_G) = \hull(S_G\setminus\{v\})\cup\{v\}$. Then, $S$ is a hull set of $H$ and then $\hncc(H)\leq |S| = |S_G|-1=\hncc(G)-1$.
\end{proof}

Corollary~\ref{cor:removesingletoncoconvex} can be used in an algorithm to compute $\hncc(G)$ by iteratively removing co-convex sets with a single vertex. We use this idea in Section~\ref{sec:poly}.

Propositions~\ref{prop:coconvex_singleton} and~\ref{prop:nocriticalvertex} also allow us to characterize graphs $G$ satisfying $\hncc(G)=n(G)$.
\begin{corollary}
\label{cor:hulltrees}
Let $G$ be a graph. Then, $\hncc(G)=n(G)$ if, and only if, $G$ is a forest.
\end{corollary}
\begin{proof}
Suppose first that $G$ has a cycle $C$. Then, for any $w\in V(C)$ we have that $\{w\}$ is not co-convex. Thus, by Proposition~\ref{prop:nocriticalvertex}, we deduce that there exists a minimum hull set of $G$ that does not contain $w$, i.e. $\hncc(G)<n(G)$.

In case $G$ has no cycles, Proposition~\ref{prop:coconvex_singleton} implies that $\{v\}$ is co-convex, for each $v\in V(G)$. Consequently, each vertex $v\in V(G)$ must belong to any hull set of $G$.
\end{proof}

In the other extreme, it is straightforward to prove that:

\begin{proposition}
\label{prop:complete}
If $G$ is a simple graph and $K$ is a clique of $G$ with at least two vertices, then $K\subseteq \icc(\{u,v\})$, for any $u,v\in K$. In particular, $\hncc(K_n)= \min\{2,n\}$.
\end{proposition}

These simple results on trees and complete graphs can be generalized to chordal graphs as we show in the sequel.

\begin{lemma}\label{lem:blocks}
Let $G$ be a connected simple graph, $u$ be a cut-vertex of $G$, and $C$ be the vertex set of some component of $G-u$. Also, let $G_1 = G[C\cup \{u\}]$ and $G_2 = G-C$. Then,
\[\hncc(G_1) +\hncc(G_2) - 1\le \hncc(G)\le \hncc(G_1)+ \hncc(G_2).\]
\end{lemma}
\begin{proof} First of all, observe that the vertex set of any cycle of $G$ is included either in $V(G_1)$ or in $V(G_2)$ ($u$ is the unique vertex in $V(G_1)\cap V(G_2)$ and $G$ has no loops as it is simple).
To proof the upper bound, it suffices then to observe that the union of hull sets of $G_1$ and $G_2$ is a hull set of $G$.

For the lower bound, consider a minimum hull set $S$ of $G$ and denote $(V(G_i)\cap S)\setminus\{u\}$ by $S_i$, for $i\in\{1,2\}$. First, note that $S_i\cup\{u\}$ is a hull set of $G_i$, for each $i\in\{1,2\}$. This means that $\hncc(G_i) \le |S_i|+1$, for each $i\in\{1,2\}$. We consider two cases.

First suppose that $u\notin \hull(S_i)$, for both for $i=1$ and $i=2$. Since $S$ is a hull set of $G$ and any cycle of $G$ is included either in $G_1$ or in $G_2$, we deduce that $u\in S$. 
In this case, note that $\hncc(G) = |S| = 1+|S_1|+|S_2|\ge \hncc(G_1)+\hncc(G_2)-1$.

Suppose then w.l.o.g. that $u\in \hull(S_1)$. Then, $u\notin S$ as otherwise $S\setminus\{u\}$ would be a smaller hull set of $G$. Also, note that $S_1$ must be a minimum hull set of $G_1$, i.e., $|S_1| = \hncc(G_1)$. In fact, observe that if $S'$ is a hull set of $G_1$, then $S'\cup S_2$ is a hull set of $G$. Thus, the existence of a smaller hull set $S_1'$ of $G_1$ such that $|S_1'|<|S_1|$ would contradict the minimality of $S$. Then, we deduce that
$\hncc(G) = |S| = |S_1|+|S_2| \ge \hncc(G_1)+\hncc(G_2)-1$ (recall that $|S_2|\ge \hncc(G_2)-1$). 
\end{proof}

\begin{corollary}\label{cor:hngep+1}
If $G$ is a connected simple graph on $n\ge 2$ vertices with $p$ blocks, then $\hncc(G)\ge p+1$.
\end{corollary}
\begin{proof}
We prove by induction on the number of blocks of $G$. If $p=1$, the corollary holds trivially since $n\ge 2$ and $G$ has no multiple edges (i.e. no singleton can be a hull set of $G$).

Suppose now that the statement holds for every connected simple graph $G'$ with less than $p$ blocks and let $G$ be a connected simple graph with $p\geq 2$ blocks. Let $G_1$ be a leaf block of $G$ connected to the rest of the graph $G$ by the cut-vertex $u$. Let $G_2 = G-V(G_1-u)$. Note that the number of blocks in $G_2$ equals $p-1$. By induction hypothesis, we then get $\hncc(G_2)\ge p$. The corollary follows by Lemma~\ref{lem:blocks} and the fact that $\hncc(G_1)\ge 2$.
\end{proof}

\begin{theorem}
\label{thm:chordal}
If $G$ is a connected simple chordal graph on $n\geq 2$ vertices with $p$ blocks, then $\hncc(G) = p+1$.
\end{theorem}
\begin{proof}
We construct a hull-set of $G$ of size $p+1$, and equality follows by Corollary~\ref{cor:hngep+1}. In~\cite{CDPRS.11}, the authors prove that if $H$ is a 2-connected chordal graph, then $\{u,v\}$ is a hull-set in the $P_3$ convexity of $H$ for any $uv\in E(H)$. Note that, since $\{u,v\}$ is initially connected, after each application of $\icc$ the obtained set is still connected. This implies that: (*) $\{u,v\}$ is a also a convex hull in the cycle convexity, for every edge $uv\in E(G)$, when $G$ is a 2-connected chordal graph. This is also used below when $G$ is not 2-connected.


Let $T$ be the block-cutpoint tree of $G$, and root $T$ at a leaf $r$. We say that the \emph{level} of a vertex $v \in V(T)$ is its distance to $r$ in $T$. Note that a vertex $b\in V(T)$ corresponding to a block $B$ must be at an even level, and we also say that the level of the block $B$ is the same of $b$. For two blocks $B_1$ and $B_2$ sharing a cut-vertex $x$, we say that $B_1$ is the \emph{parent} of $B_2$ if the level of $B_1$ is equal to the level of $B_2$ minus two, and we call $x$ the \emph{root} of $B_2$.

Let $B_r$ be the leaf-block corresponding to the root $r$ and having $x_r$ as unique cut-vertex.
We add $x_r$ to $S$ and we also arbitrarily choose one neighbor of $x_r$ in $B_r$ to add to $S$.
Then, for each block $B$ distinct from $B_r$, let $x$ be the root of $B$ and choose a neighbor of $x$ in $B$ to add to $S$. Note that we have chosen exactly $p+1$ vertices. It remains to prove that $S$ is a hull set of $G$.

Since $B_r$ is 2-connected and $x_r,v\in S$ for some $v\in N(x_r)\cap B_r$, by (*) we get that $V(B_r)\subseteq \icc(S)$. Suppose then that, for each block $B$ at level $2(k-1)$, we have that $V(B)\subseteq \iccp{k}(B)$, $k\geq 1$. Let $B$ be a block at level $2k$. By hypothesis, the root $x$ of $B$ belongs to $\iccp{k}(S)$, and since we chose a neighbor of $x$ in $B$ to add to $S$, by (*) we deduce that $V(B)\subseteq \iccp{k+1}(S)$.
Since the graph is finite, it follows that $S$ is a hull set of $G$, as we wanted to prove.
\end{proof} 

Another graph class that we can present an explicit formula to compute the hull number in the cycle convexity is the one of grids. Let us first analyse the structure of co-convex sets in this class.

Let $G= G_{m\times n}$ be a grid. We say that a subset $S\subseteq V(G)$ is \emph{boxed} if every connected component of $G[S]$ is also a grid. Let us define $L_i = \{(i,j)\mid j\in\{1,\ldots,n\}\}$ to be the \emph{$i$-th row of} $G_{m\times n}$ and $C_j = \{(i,j)\mid i\in\{1,\ldots,m\}\}$ to be its \emph{$j$-th column}.

\begin{proposition}\label{prop:squared}
Let $G= G_{m\times n}$ be a grid. $S\subseteq V(G)$ is convex if, and only if, $S$ is boxed.
\end{proposition}
\begin{proof}
Suppose first that $S$ is convex.
Let $C$ be a connected component of $G[S]$. As a single vertex is considered to be a grid, we get that if $C$ is not a grid, then $G$ must have four vertices $(i,j)$,$(i+1,j)$,$(i,j+1)$ and $(i+1,j+1)$ such that exactly one of them, say $(i,j)$ does not belong to $C$.
However, note that these four vertices form a cycle, thus $(i,j)$ should belong to $S$ (and thus to $C$), since $S$ is convex, a contradiction.

On the other hand, suppose that $S$ is boxed. Let $C$ be a connected component of $G[S]$. Note that no vertex in $V(G)\setminus V(C)$ has two distinct neighbors in $C$. Consequently, $\icc(S)=S$, i.e. $S$ is convex.
\end{proof}

Thus, observe that each co-convex set $S$ of a grid $G = G_{m\times n}$ is obtained from $G$ by the removal of a boxed subset $V(G)\setminus S$.

\begin{corollary}
\label{cor:linescolumnscoconvex}
Each line and each column of $G_{m\times n}$ is a co-convex set.
\end{corollary}
\begin{proof}
It follows directly from Proposition~\ref{prop:squared} as $V(G_{m\times n})\setminus L_i$ and $V(G_{m\times n})\setminus C_j$ are boxed sets, for every $i\in\{1,\ldots,m\}$ and every $j\in\{1,\ldots,n\}$.
\end{proof}

Although we know how to describe precisely the (co-)convex sets of a grid, in order to determine the precise value of its hull number one might struggle to prove a lower bound by using such co-convex sets. We use a similar approach as the folklore one to compute the hull number of grids in the $P_3$-convexity that uses the notion of perimeter.

The idea is to consider the grid $G_{m\times n}$ as being the graph representation of the spaces on a chessboard $H$, i.e. each vertex of $G$ represents a space in $H$ and two vertices are adjacent if the spaces share a boundary. In what follows, we treat the grid as a plane graph, i.e., when considering the grid, we are also assuming a drawing in the plane.

Given a plane graph $G$ with face set $F=F(G)$, the \emph{dual} of $G$ is the plane graph $G^* = (F,E^*)$ where $ff'\in E^*$ if and only if there exists $e\in E(G)$ such that $f$ is the face in one side of $e$ and $f'$ is the other (note that $f$ and $f'$ might be the same, in which case we create loops in $G^*$). Observe that there is a bijection between $E(G^*)$ and $E(G)$. Hence, given $e\in E(G)$, we denote by $e^*$ the edge of $G^*$ related to $e$. Similarly, given an edge $e\in (G^*)$, we denote by $e^*$ the edge of $G$ related to $e$.

Now, let $H$ be an $(m+1)\times(n+1)$ grid and $G$ be its weak dual, i.e the graph obtained from the dual graph $G^*$ by the removal of the vertex corresponding to the outer face of $G$. Observe that $G$ is the $m\times n$ grid. Given a subset of edges $X\subseteq E(G)$, we denote by $E^*(X)$ the set of corresponding edges $\{e^*\in E(H)\mid e\in X\}$. Given
 a subset of vertices $S\subseteq V(G)$, we also define the \emph{boundary of $S$}, denoted by $B(S)$, as the set edges of $H$ that bounds $S$. More formally, $B(S)$ is the set of edges that belong to some face of $H-E(G[S])$ that contains some vertex of $S$. The \emph{perimeter of $S$} is then defined as $\per(S) = |B(S)|$.
 
 For instance, if $S$ contains a single vertex, then $\per(S)=4$, and if $S = V(G)$, then the perimeter of $S$ is equal to the number of edges in the outer face of $G$, i.e. $\per(S) = 2m+2n$. The following is the key lemma.

\begin{lemma}\label{lem:grid}
Let $H$ be an $(m+1)\times(n+1)$ grid, $G$ be its weak dual and $S\subseteq V(G)$. Then, $\per(\icc(S))\le \per(S)$ and $\per(\hull(S))\le \per(S) - 2c$, where $c$ is the number of  components of $G[S]$ minus the number of components of $G[\hull(S)]$. 
\end{lemma}
\begin{proof}
If $\icc(S) = S$, then there is nothing to prove. So suppose otherwise and write $\icc(S)\setminus S$ as $\{v_1,\ldots,v_p\}$. We prove the first part by induction on $p$. Suppose it holds for $\ell-1$, i.e., that $\per(S') \le \per(S)$, where $S' = S\cup \bigcup_{i=1}^{\ell-1}\{v_i\}$, and let $x,y\in S$ be neighbors of $v_\ell$ in a same connected component of $G[S]$. Let $X = E(G[S])$, and note that each connected component of $G[S]$ is contained in the same face of $H-E^*(X)$. Indeed, if $uv\in E(G[S])$, then by removing the edge related to $uv$ from $H$, we are merging the faces of $H$ related to $u$ and $v$. Note also that, because $v_\ell\notin S'$, none of the edges related ot its incident edges is removed when computing the perimeter of $S'$, which means that $v_\ell$ is not inside the boundary of $S'$, and that $xv_\ell$ and $yv_\ell$ are counted in $\per(S')$. Now, let $S'' = S'\cup \{v_\ell\}$. We get that the edges related to $xv_\ell,yv_\ell$ are within $E^*(E(G[S'']))$ and hence are not in $B(S'')$. Therefore, we lose at least two edges in the boundary, and gain at most two edges, namely when the other two neighbors of $v_\ell$ are not in $S'$. 
Finally, note that if $v_\ell$ has at least~3 neighbors in $S'$, then the perimeter decreases by at least two, since we will have that at least 3 edges incident in $v_\ell$ are in the boundary of $S'$, and at most 1 is in the boundary of $S''$. Because the only way to connect the components of $G[S]$ in the hull of $S$ is if that happens at least $c$ times, the second part of the lemma follows.
\end{proof}

We are now ready to prove our theorem.

\begin{theorem}
Let $G$ be a $m\times n$ grid. Then, $\hncc(G) = m+n-1$.
\end{theorem}
\begin{proof}
First, note that if $S$ contains the first row and first column of  vertices, then $\hull(S) = V(G)$; hence $\hncc(G)\le m+n-1$. To prove that this is also a lower bound, first we compute the maximum perimeter of a subset $S$. Let $X=E(G[S])$ and $\delta(S)$ denote the set of all edges incident to $S$. Observe that the edges in the boundary of $S$ are exactly the edges in $E^*(\delta(S))$ that are not in $E^*(X)$. Applying also Lemma~\ref{lem:grid}, and letting $c$ be the number of components of $S$, we get that, if $S$ is a hull set of $G$, then:
\[ \begin{array}{lll} \per(\hull(S)) & \le & (4|S| - 2|E^*(X)|) - 2 (c-1) \\
                  & =  & 4|S| - 2|E(G[S])| -2c+2 \\
                  & \le & 4|S|-2(|S|-c) -2c+2\\
                  & = & 2|S|+2.\end{array}\]

Now, since $\per(V(G)) = 2m+2n$, we know that if $S$ is a hull set of $G$, then \[2|S|+2\ge \per(\hull(S)) = \per(V(G)) = 2m+2n.\]
We then get $|S|\ge m+n-1$, as we wanted to prove.
\end{proof}


\section{Upper Bound for $\mathbf{4}$-Regular Planar Graphs}\label{sec:upperbound}

The idea to prove such bound is to iteratively apply operations in the input graph $G$ that ensure at each step the properties: the vertex set is reduced by at least two vertices; the size of a minimum hull set decreases by at at most one; the operations maintain the properties of being 4-regular and planar; if no operation can be applied, then the obtained graph $G'$ satisfies $\hncc(G') = \frac{1}{2}n(G')$.

The main result of this section concerns 4-regular planar graphs, but our first lemma actually hold for any 4-regular graph.


Let $G$ be any 4-regular graph. Let us introduce the reduction operations. 

\begin{itemize}

 \item[M4] Let $uv$ be such that $\mu(uv)=4$. Since $G$ is 4-regular, note that $G$ has component with vertex set $\{u,v\}$. Unless $V(G)=\{u,v\}$, remove $u$ and $v$ from $G$.






  \item[M2] Let $H$ be a component of the subgraph of $G$ containing exactly the edges of multiplicity 2. Because $G$ is 4-regular and $\mu(uv)=2$ for every $uv\in E(H)$, we get that $H$ is either a cycle or a path. In case it is a cycle, note that it is a component of $G$. Remove all vertices in $V(H)$ from $G$.
  In case $H$ is a path between vertices $u$ and $v$ such that $N_{G-H}(u) = \{x,x'\}$ and $N_{G-H}(v) = \{y,y'\}$. Then add edges $xx'$ and $yy'$ (see Figure~\ref{fig:M2}).

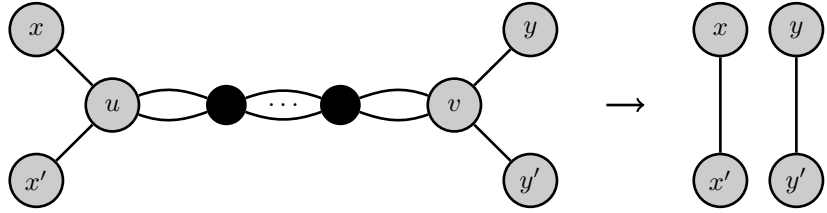
\begin{figure}[thb]
\centering
  \begin{tikzpicture}[scale=1]
  \pgfsetlinewidth{1pt}
  \tikzset{vertex/.style={circle, minimum size=0.7cm, fill=black!20, draw, inner sep=1pt}}
  \tikzset{simple/.style={circle, minimum size=0.5cm, fill=black, draw, inner sep=1pt}}

  \node [vertex] (u) at (-1,1){$u$};
  \node [vertex] (v) at (3.5,1){$v$};
  \node [vertex] (x) at (-2,2){$x$};
  \node [vertex] (xp) at (-2,0){$x'$};
  \node [vertex] (y) at (4.5,2){$y$};
  \node [vertex] (yp) at (4.5,0){$y'$};
  \node [simple] (v1) at (0.5,1) {};
  \node at (1.25,1) {$\ldots$};
  \node [simple] (v2) at (2,1) {};
   \draw [-] (u) to [out=20,in=160] (v1)
   		(u) to [out=-20,in=-160] (v1)
   		(v1) to [out=20,in=160] (v2)
   		(v1) to [out=-20,in=-160] (v2)
   		(v2) to [out=20,in=160] (v)
   		(v2) to [out=-20,in=-160] (v);
  \draw (xp)--(u)--(x) (yp)--(v)--(y);
  
  \path [->] (5.5,1) edge node {} (6,1);

  \node [vertex] (xn) at (7,2){$x$};
  \node [vertex] (xpn) at (7,0){$x'$};
  \node [vertex] (yn) at (8,2){$y$};
  \node [vertex] (ypn) at (8,0){$y'$};
   \draw (xn)--(xpn) (yn)--(ypn);
  \end{tikzpicture}

\caption{Operation M2 in case $H$ is a path (if $H$ is a cycle, we just remove $V(H)$).}
\label{fig:M2}
\end{figure}

  \item[M3] Let $uv\in E(G)$ be such that $\mu(uv) = 3$. Also let $x\in N(u)\setminus\{v\}$ and $y\in N(v)\setminus \{u\}$. If $x\neq y$, remove vertices $u$ and $v$ and add edge $xy$ to $G$ (see Figure~\ref{fig:M3xneqy}). Otherwise, let $N(x)\setminus\{u,v\} = \{z,z'\}$. Remove vertices $u,v,x$ and add edge $zz'$ (see Figure~\ref{fig:M3xeqy}). 
  
  \begin{figure}[thb]
  \centering
  \begin{tikzpicture}[scale=1]
  \pgfsetlinewidth{1pt}
  \tikzset{vertex/.style={circle, minimum size=0.7cm, fill=black!20, draw, inner sep=1pt}}
  \node [vertex] (u) at (-1,1){$u$};
  \node [vertex] (v) at (1,1){$v$};
  \node [vertex] (x) at (-2,1){$x$};
  \node [vertex] (y) at (2,1){$y$};
   \draw [-] (u) to [out=30,in=150] (v)
   		(u) to [out=-30,in=-150] (v)
   		(u) to [out=0,in=180] (v)
   		(x) to (u)
   		(y) to (v); 		

\path [->] (3,1) edge node {} (3.5,1);
  \node [vertex] (x) at (4.5,1){$x$};
  \node [vertex] (y) at (5.5,1){$y$};
   \draw [-]  (x) to (y);  		
  \end{tikzpicture}

\caption{Operation M3 when $x\neq y$.}
\label{fig:M3xneqy}
\end{figure}
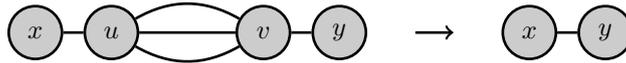

\begin{figure}[thb]
\centering
  \begin{tikzpicture}[scale=1]
  \pgfsetlinewidth{1pt}
  \tikzset{vertex/.style={circle, minimum size=0.7cm, fill=black!20, draw, inner sep=1pt}}
  \node [vertex] (u) at (1,-1){$u$};
  \node [vertex] (v) at (1,1){$v$};
  \node [vertex] (x) at (-0.5,0){$x$};
    \node [vertex] (z) at (-1.5,-0.5){$z$};
        \node [vertex] (zp) at (-1.5,0.5){$z'$};
   \draw [-] (u) to [bend left] (v)
   		(u) to (v)
   		(u) to [bend right] (v);
  \draw (z)--(x)--(zp) (u)--(x)--(v); 
  
  \path [->] (2,0) edge node {} (2.5,0);
  
  \node [vertex] (zn) at (3.5,-0.5){$z$};
  \node [vertex] (zpn) at (3.5,0.5){$z'$};
   \draw [-]  (zn) to (zpn);		
  \end{tikzpicture}

\caption{Operation M3 when $x= y$.}
\label{fig:M3xeqy}
\end{figure}
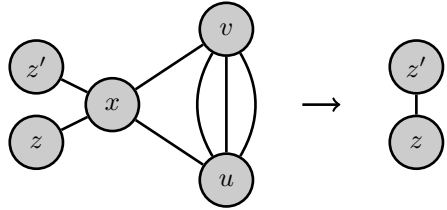

   \item[T] Let $T=(x,y,z)$ be any cycle of length~3 in $G$, and let $N_{G\setminus T}(x) = \{x_1,x_2\}$, $N_{G\setminus T}(y) = \{y_1,y_2\}$, and $N_{G\setminus T}(z) = \{z_1,z_2\}$. Remove $V(T)$ from $G$, add an edge $x_1x_2$, a vertex $w$ and edges $wy_1,wy_2,wz_1,wz_2$ to $G$ (see Figure~\ref{fig:T}).
   
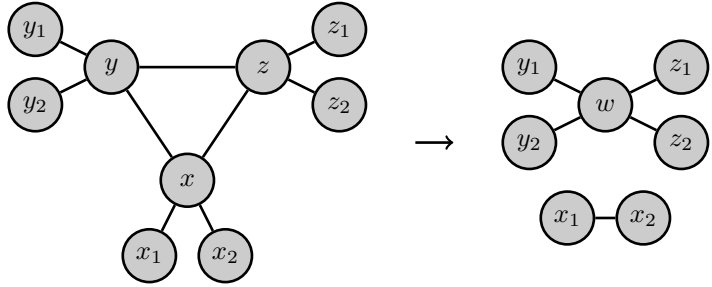
\begin{figure}[thb]
\centering
  \begin{tikzpicture}[scale=1]
  \pgfsetlinewidth{1pt}
  \tikzset{vertex/.style={circle, minimum size=0.7cm, fill=black!20, draw, inner sep=1pt}}
  \node [vertex] (y) at (-1,1){$y$};
  \node [vertex] (z) at (1,1){$z$};
  \node [vertex] (x) at (0,-0.5){$x$};
   \node [vertex] (y1) at (-2,1.5){$y_1$};
   \node [vertex] (y2) at (-2,0.5){$y_2$};
   \node [vertex] (z1) at (2,1.5){$z_1$};
   \node [vertex] (z2) at (2,0.5){$z_2$};
   \node [vertex] (x1) at (-0.5,-1.5){$x_1$};
   \node [vertex] (x2) at (0.5,-1.5){$x_2$};
   \draw (x)--(y)--(z)--(x) (y1)--(y)--(y2) (z1)--(z)--(z2) (x1)--(x)--(x2);
  
  \path [->] (3,0) edge node {} (3.5,0);
  \node [vertex] (w) at (5.5,0.5){$w$};
   \node [vertex] (y1) at (4.5,1){$y_1$};
   \node [vertex] (y2) at (4.5,0){$y_2$};
   \node [vertex] (z1) at (6.5,1){$z_1$};
   \node [vertex] (z2) at (6.5,0){$z_2$};
   \node [vertex] (x1) at (5,-1){$x_1$};
   \node [vertex] (x2) at (6,-1){$x_2$};
   \draw (x1)--(x2) (y1)--(w)--(y2) (z1)--(w)--(z2);
  \end{tikzpicture}

\caption{Operation T.}
\label{fig:T}
\end{figure}
\end{itemize}

\begin{lemma}
\label{lem:operations4regular}
Let $G$ be a 4-regular graph and $F$ a graph obtained from $G$ by the application of one of the operations M4, M2, M3 or T. Then, $F$ is 4-regular and $\hncc(G)\le \hncc(F)+1$.
\end{lemma}
\begin{proof}
In this proof, for all the cases, we use the same notation to vertices and components as in the definitions of the corresponding operations.

Note that all operations preserve the property of being 4-regular. Consequently, we just prove the upper bound on $\hncc(G)$. Thus, let $S$ be a minimum hull set of $F$. We analyse the cases below.

Consider first that operation M4 was applied on edge $uv\in E(G)$ to obtain $F$. 
Let $S'= S\cup\{u\}$. Since $S$ is a hull set of $H$, note that $V(F)\subseteq \hull(S')$. Moreover, due to the multiple edges linking $u$ and $v$, we get that $v\in \icc(S')\subseteq \hull(S')$. Consequently, $S'$ is a hull set of $G$ and $\hncc(G)\leq |S'| = |S| +1 = \hncc(F)+1$.

Suppose now that $F$ is obtained from $G$ by operation M2. In case the component $H$ corresponds to a cycle, let $u$ be an arbitrary vertex in that cycle. Otherwise, let it be an endpoint of $H$ as in the definition of M2. Again, let ${S'= S\cup\{u\}}$.
Since $H$ only has multiple edges and $u\in S'$, note that $V(H)\subseteq \hull(S')$. In case $H$ is a cycle, it is a component of $G$ and ${V(G)\setminus V(H)}\subseteq \hull(S)\subseteq \hull(S')$. Consequently, $S'$ is a hull set of $G$ and thus $\hncc(G)\leq \hncc(F)+1$. Suppose now that $F$ is obtained by applying M2 on $G$, by considering that the removed subgraph $H$ is a path. Note that any cycle in $F$ that uses the edge $xx'$ (resp. $yy')$ has a corresponding cycle in $G$ containing the path $(x,u,x')$ (resp. $(y,v,y')$). Consequently, since $S$ is a hull set of $F$ and $\{u,v\}\subseteq V(H)\subseteq \hull(S')$, we deduce that $S'$ is a hull set of $G$.

The case in which $F$ is obtained from $G$ by an application of operation M3 is similar to the previous one. In both cases, we consider $S' = S\cup \{u\}$. Note that if $x\neq y$, then $v\in \hull(S')$ because of the multiplicity of edge $uv$. And if $x=y$ (see Figure~\ref{fig:M3xneqy}), we have that $\{x,v,u\}\in \hull(S')$ due to the fact that $u\in S'$ and that any cycle in $F$ using the edge $zz'$ has a corresponding one in $G$ passing through $x$. Thus, in both cases, $S'$ is a hull set of $G$, and $\hncc(G)\leq \hncc(F)+1$.

Finally, suppose that $F$ is obtained by applying operation T on a triangle $(x,y,z)$. Let us first argue that the new vertex $w\in V(F)$ must lie in some cycle of $F$. If not, then observe that $w$ must have four distinct neighbors and any two neighbors of $w$ must belong to distinct connected components of $F-w$. Since $F$ is 4-regular, note that each of these four connected components of $F-w$ has exactly one vertex with degree three (the neighbor of $w$) and all the other vertices have degree four, contradicting the fact that the sum of degrees of vertices in each connected component must be even.
Therefore, by Proposition~\ref{prop:coconvex_singleton}, we get that $\{w\}$ is not a co-convex set of $F$, and by Proposition~\ref{prop:nocriticalvertex}, we can suppose, w.l.o.g., that $S$ is a minimum hull set of $F$ such that $w\notin S$.

Before we proceed, observe that any cycle in $F$ containing vertices $w$, $y_1$, $y_2$, $z_1$, $z_2$, $x_1$ and $x_2$ has a corresponding one in $G$ such that: if the cycle contains the path $(y_1, w,y_2)$ (resp. $(z_1, w,z_2)$) we just replace $w$ by $y$ (resp. $z$); if it contains a path $(y_i,w,z_j)$, for some $i,j\in\{1,2\}$ then we replace $w$ in the path by $(y,z)$; and if it uses the edge $x_1x_2$, then we add $x$ to such cycle to obtain a corresponding one in $G$.

Recall that $w\notin S$, and let $k\ge 1$ be the step in which $w$ is contaminated by $S$ in $F$. It means that there exists a cycle $C$ in $F$ such that $C-w\subseteq \iccp{k-1}(S)$. Let $u$ and $v$ be the neighbors of $w$ in $C$, and let $C'$ denote the corresponding cycle in $G$ obtained as in the previous paragraph. We analyse two cases by considering whether $w$ is contaminated by $S$ in $F$ either before or after $x_1$ and $x_2$. In both cases, by adding one of the vertices $\{x,y,z\}$ to $S$, we obtain a hull set $S'$ of $G$, thus satisfying that $\hncc(G)\leq |S'|=|S|+1 = \hncc(F)+1$. 

\begin{itemize}
\item Neither $x_1$ nor $x_2$ is contaminated by $S$ in $F$ up to step $k-1$: 
this implies that $C'\setminus \{y,z\}\subseteq \iccp{k-1}(S)$. If $\{u,v\} = \{y_1,y_2\}$, let $S' = S\cup \{z\}$, and if $\{u,v\} = \{z_1,z_2\}$, let $S' = S\cup \{y\}$. Observe that in the former case, we get that $C'$ is a cycle containing $y$ and not $z$, which implies that $y\in \iccp{k}(S)$. A similar argument holds for the latter case, thus giving us that $\{y,z\}\subseteq \iccp{k}(S')$. And if $\{u,v\} = \{y_i,z_j\}$ for some $i,j\in\{1,2\}$, let $S' = S\cup \{y\}$ (we could also add $z$ instead). Because $C'$ in this case is a cycle containing both $y$ and $z$, and $C'\setminus \{z\}\subseteq \iccp{k-1}(S')$, we again get that $\{y,z\}\subseteq \iccp{k}(S')$.

Therefore we also get $x\in\iccp{k+1}(S')$. Since $S$ is a hull set of $F$, and because the vertices in the triangle $(x,y,z)$ can be used in the cycles of $G$ corresponding to the cycles in $F$ as we previously mentioned, we get that $S'$ is a hull set of $G$.

\item $x_1$ is contaminated by $S$ in $F$ at step $p_1\leq k-1$: 
if $x_2$ is also contaminated by $S$ in $F$ before $w$, then note that $x\in\iccp{k}(S)$. Consequently, we can chose $S'$ as in the previous case and a similar argument holds. Thus, consider that $x_2$ is contaminated at step $p_2\geq k$.

Since $w,x_2\notin \iccp{k-1}(S)$, observe that the vertices contaminated by $S$ in $F$ are also contaminated by the same subset $S$ in $G$ up to step $k-1$. If the cycle $C$ (used to contaminate $w$) does not contain $x_1$, then again we chose $S'$ as in the previous case and a similar argument holds.

Otherwise, let $S'=S\cup\{x\}$. Recall that the unique vertex of $C$ that does not belong to $\iccp{k-1}(S)$ is $w$. Note that $C$ is composed of two $w,x_1$-paths $P_1$ and $P_2$ and suppose that $u\in V(P_1)$ and $v\in V(P_2)$. Suppose that $u=y_1$ (the other cases are similar). If we replace $P_1$ by the path $(y_1,y,x,x_1)$, note that we obtain a cycle $C'$ such that $y$ is the unique vertex of $C'$ that might not belong to $\iccp{k-1}(S')$. Consequently, $y\in \iccp{k}(S')$ and $z\in \iccp{k+1}(S')$. Again, since $S$ is a hull set of $F$, we may use the corresponding cycles to deduce that $S'$ is a hull set of $G$.
\end{itemize}
\end{proof}

Now suppose that $G$ is a 4-regular plane graph.
Let $G'$ be obtained from $G$ by exhaustively applying operations M4, M2, M3, and T, with this priority order. It means that, at each step, we just apply M2 if M4 cannot be applied; we just apply M3 if M4 and M2 cannot be applied; and we just apply T if M4, M2 and M3 cannot be applied. Because $G$ is finite and since each operation decreases the number of vertices of the obtained graph, we get that this process finishes. We call $G'$ a \emph{reduction graph of $G$} and we say that $G$ \emph{can be reduced} to $G'$. 

\begin{lemma}\label{lem:Gprime}
Let $G$ be a 4-regular plane graph and $G'$ be its reduction graph. Then, $G'$ is the 4-regular graph on two vertices.
\end{lemma}
\begin{proof}
First, note that operations M4, M2, M3 and T not only preserve the property of being 4-regular, as claimed in Lemma~\ref{lem:operations4regular}, but also planarity. While this can be easily verified for operations M4, M2 and M3, the case of operation T is a bit less trivial if the triangle $(x,y,z)$ is not a face of $G$. However, note that operation $T$ can be seen as we first contract the edges $yz$, $x_1x$ and $x_2x$ and then we remove multiple edges that might appear. Thus, operation T constructs a minor of $G$, which must be planar since planar graphs are minor-closed. So, $G'$ is a planar 4-regular graph.

Note also that, unless $G'$ has exactly two vertices (because of operation M4), $G'$ cannot have any edges of multiplicity bigger than~1, as otherwise we could apply operations M4, M2 or M3. Moreover, $G'$ cannot have any cycle of length three, as otherwise we could still apply operation $T$. Consequently, note that any planar embedding of $G'$ has no face of degree two or three.

We also argue that these operations cannot create loops, i.e. that any planar embedding of $G'$ cannot have faces of degree one. Recall that we start from a loopless graph $G$. Operation M4 clearly does not create any loop, since it does not add any edges. Note that operation M2 adds edges $xx'$ and $yy'$ that cannot be loops because, otherwise, $ux$ or $vy$ would have multiplicity two and, therefore, they should also be in $H$. The case of applying operation M3, when $x=y$ is the only case in which we could have added a loop within this operation, but then we note that $z\neq z'$ as, otherwise, $xz$ would be an edge of multiplicity two and we should apply M2 instead, because of the priority order. A similar argument is applied to see that $x_1x_2$ is not a loop in operation T.

We then get that either $G'$ has exactly two vertices linked by an edge of multiplicity four or it must be a simple 4-regular planar graph with no cycles of length three. Let us argue that the latter case may not happen due to Euler's Equation.

By contradiction, suppose that the reduction graph $G'$ is a simple 4-regular planar graph having no cycles of length three. Let $C$ be a planar embedding of a connected component of $G'$. Consequently, note that $C$ is a connected simple 4-regular plane graph having no face of degree three.

So, let $n=\lvert V(C)\rvert$, $m=\lvert E(C)\rvert$, $f$ denote the number of faces of $C$, and let $f_i$ denote the number of faces of degree $i$ in $C$. Since $C$ is 4-regular, we have that $4n=\sum_{v\in V(C)} d(v) = 2m$. Then, we can use Euler's Equation and the fact that $C$ has no faces of degree one, two or three, to deduce that:
\[n+f-m = 2 \Rightarrow n = f - 2 = \left(\sum_{i\ge 1}f_i\right) - 2= \left(\sum_{i\ge 4}f_i\right) - 2.\]
	
Since each edge is counted twice when summing up the degrees of faces of $C$, we can use the previous equation to deduce that:
\[\sum_{i\ge 4}(i\cdot f_i) = 2m = 4n = 4\left(\sum_{i\ge 4}f_i\right) - 8.\]
This is a contradiction as:
\[\sum_{i\ge 4}(i\cdot f_i)>4\left(\sum_{i\ge 4}f_i\right) - 8.\]
\end{proof}

We are now ready to present our main result.

\begin{theorem}\label{theo:4regular}
If $G$ is a 4-regular planar graph, then:
\[\hncc(G)\le \frac{1}{2}\lvert V(G)\rvert.\] Furthermore, this bound is tight.
\end{theorem}
\begin{proof}
By Lemma~\ref{lem:operations4regular}, whenever we apply one of the operations M4, M2, M3 or T, we decrease the vertex set of the input graph by at least two units, while the value of the hull number of the obtained graph decreases by at most one unit. Since, by Lemma~\ref{lem:Gprime}, the reduction graph of $G$ is the 4-regular graph on two vertices $G'$ and $\hn(G')=1$, the upper bound follows. The bound is attained when $G$ has exactly two vertices, for instance.
\end{proof}

One should notice that, by analysing all cases in Lemma~\ref{lem:operations4regular}, one can actually build a hull set of $G$ from a hull set of its reduction $G'$.

Let us finally present a more general (asymptotically) tight example to such bound (see Figure~\ref{fig:tight}). It consists of the graph $G_k$, for an integer $k\geq 2$, obtained from two disjoint cycles of length $k$, $(u_1,\ldots,u_k)$ and $(v_1,\ldots,v_k)$, by adding edges of multiplicity two $u_iv_i$ for each $i\in \{1,\ldots,k\}$.
Let $S$ be a minimum hull set of $G_k$. We shall prove that $|S|=\hncc(G_k)= k-1$. Let us first prove that $|S|\geq k-1$. Note that if $|S|<k-1$, by the pigeonhole principle there exist $i,j\in\{1,\ldots,k\}$ such that $\{u_i,v_i,u_j,v_j\}\cap S=\emptyset$. Let $X=\{u_i,v_i,u_j,v_j\}$. Note that $X$ is a co-convex set of $G$ as $G$ has no cycle such that the unique vertex not in $V(G)\setminus X$ is in $X$. This contradicts the fact that $S$ is a hull set of $G_k$. On the other hand, one can easily check that $\{u_i\mid i\in\{1,\ldots,k-1\}\}$ is a hull set of $G_k$. Therefore, we deduce that:
\[\frac{\hncc(G_k)}{n(G_k)} = \frac{\hncc(G_k)}{2k}=\frac{k-1}{2k}\rightarrow\frac{1}{2}\]
as $k\rightarrow+\infty.$

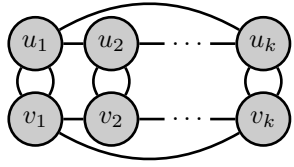
\begin{figure}[thb]
\centering
  \begin{tikzpicture}[scale=1]
  \pgfsetlinewidth{1pt}
  \tikzset{vertex/.style={circle, minimum size=0.7cm, fill=black!20, draw, inner sep=1pt}}
  \node [vertex] (u1) at (0,1){$u_1$};
  \node [vertex] (v1) at (0,0){$v_1$};
  \node [vertex] (u2) at (1,1){$u_2$};
  \node [vertex] (v2) at (1,0){$v_2$};
  \node [vertex] (up) at (3,1){$u_k$};
  \node [vertex] (vp) at (3,0){$v_k$};
  
  \node (rt1) at (2,0){$\ldots$};
  \node (rt1) at (2,1){$\ldots$};
  
  \draw (u1) -- (u2) (v1) -- (v2) (u2) -- (1.7,1) (v2) -- (1.7,0) (2.3,1) -- (up) (2.3,0) -- (vp);
  \path [-] (u1) edge [bend left] node {} (v1);
  \path [-] (u1) edge [bend right] node {} (v1);
  \path [-] (u2) edge [bend left] node {} (v2);
  \path [-] (u2) edge [bend right] node {} (v2);
  \path [-] (up) edge [bend left] node {} (vp);
  \path [-] (up) edge [bend right] node {} (vp);
  
  \path [-] (u1) edge [bend left] node {} (up);
  \path [-] (v1) edge [bend right] node {} (vp);
  
  \end{tikzpicture}

\caption{Asymptotically tight example to Theorem~\ref{theo:4regular}.}
\label{fig:tight}
\end{figure}

\section{\NP-Completeness}\label{sec:npcomplete}

Let us formally define the decision version of the problem we study:

\begin{problem}{\textsc{Hull Number in Cycle Convexity}}\

	\textbf{Input: } A graph $G$ and a positive integer $k$.

	\textbf{Question: } Is $\hncc(G)\leq k$?
\end{problem}

The goal of this section is to prove that:

\begin{theorem}
\label{thm:np-c}
	\textsc{Hull Number in Cycle Convexity} is \NP-Complete, even if the input graph $G$ is a simple planar graph. 
\end{theorem}

We shall reduce a variant of \textsc{Planar 3-\SAT} to \textsc{Hull Number in Cycle Convexity}. We first explain which variant we use.


Let $\varphi$ be a boolean formula in 3-\CNF having a set $\mathcal C$ of $m$ clauses over the variables $\mathcal{X} = \{x_1,\ldots, x_n\}$. 
The \emph{$\varphi$-graph} is the graph $G(\varphi)$ such that:
\begin{itemize}
	\item $V(G(\varphi)) = \mathcal{X}\cup \mathcal{C}\cup\{\bar{x}_1,\ldots,\bar{x}_n\}$; and
	\item $E(G(\varphi)) = E_1\cup E_2\cup E_3$, where:
	\begin{itemize}
		\item $E_1 = \{x_i C\mid x_i \text{ appears positively in } C\}$, 
		\item $E_2 = \{\bar{x}_i C\mid x_i \text{ appears negatively in } C\}$, and
		\item $E_3 = \{x_1x_2, x_2x_3, \ldots, x_nx_1, x_1\bar{x}_1,\ldots,x_n\bar{x}_n\}$.
	\end{itemize}
\end{itemize}

We say that $\varphi$ is \emph{linkable} if it is possible to add edges to $G(\varphi)$ in such a way that the obtained graph is planar and that the subgraph induced by $\mathcal{C}$ is connected. We call the obtained graph a \emph{linked $\varphi$-graph}, and call the subgraph formed by the added edges a \emph{link}. Observe that $G(\varphi)$ must be planar itself to start with, and that we can consider that the link only contains edges between vertices in $\mathcal{C}$ as otherwise we could remove some of the edges and still have a linked $\varphi$-graph.

The variation of the \textsc{3-\SAT} problem presented next can be proved to be NP-complete by combining and making slight modifications in the proofs presented in~\cite{Lichtenstein82,Tippenhauer2016,MG08}. Because the needed ideas are not in any way new, we restrain from presenting them here.

\begin{problem}{\textsc{Linked Planar exactly 3-bounded 3-\SAT}}\

	\textbf{Input: } A linkable boolean formula $\varphi$ in 3-\CNF such that each variable of $\varphi$ appears exactly three times: twice positively and once negatively.

	\textbf{Question: } Is $\varphi$ satisfiable?
\end{problem}

In what follows, we present constructions for edge, variable and clause gadgets that will replace the respective structures in a linked $\varphi$-graph $G(\varphi)$, obtaining a planar graph $H$. Then, we prove the formula $\varphi$ on $n$ variables and $m$ clauses is satisfiable if, and only if, $\hncc(H)\le 4n+4m-4$.

Let $G$ be a connected graph and $H\subseteq G$ be an induced subgraph of $G$. We say that $H$ is \emph{pendant at} $v\in V(H)$ if there are no edges between $V(H-v)$ and $V(G-H)$ (in other words, $v$ separates $H$ from the rest of the graph $G$). To construct our variable and clause gadgets, we first need what we call \emph{auxiliary gadget}.
This is the graph $\Lambda$ with vertex set $\{a,b,c,d,e\}$  depicted in Figure~\ref{fig:aux-clause-gadgets}.(a). The dotted edges show how these graphs will be appended to the rest of the constructed graph. 

\begin{figure}
	\centering
	\subfloat[Auxiliary gadget $\Lambda$.]{
	\begin{tikzpicture}[scale=1]
  \pgfsetlinewidth{1pt}
  \tikzset{vertex/.style={circle, minimum size=0.7cm, fill=black!20, draw, inner sep=1pt}}
  \node [vertex] (a) at (-1,1){$a$};
  \node [vertex] (b) at (1,1){$b$};
  \node [vertex] (c) at (0,0){$c$};
  \node [vertex] (d) at (-1,-1){$d$};
  \node [vertex] (e) at (1,-1){$e$};
  \node [vertex,dashed] (n) at (0,2){};
  
  \draw (a) -- (c) -- (b) -- (e) -- (d) -- (a)   (d) -- (c) (c) -- (e);
  \path [-] (a) edge [dashed] node {} (n);
  \path [-] (b) edge [dashed] node {} (n);
  \end{tikzpicture}
	}\hspace{1cm}
	\subfloat[Edge gadget $\Xi_u^v$.]{
	\begin{tikzpicture}[scale=1]
  \pgfsetlinewidth{1pt}
  \tikzset{vertex/.style={circle, minimum size=0.7cm, fill=black!20, draw, inner sep=1pt}}
  \node [vertex] (ci) at (0,0){$c_i$};
  \node [vertex] (r) at (2,0){$r_{i,j}$};
  \node [vertex] (s) at (3,0){$s_{i,j}$};
  \node [vertex] (cj) at (5,0){$c_j$};
  \node [vertex] (q) at (1,1){$q_{i,j}$};
  \node [vertex] (t) at (4,1){$t_{i,j}$};
  
  \draw (ci) -- (r) -- (s) -- (cj) -- (t) -- (s)    (c) -- (q) -- (r);
  
  \def\xpos{1}
  \def\ypos{2.3}
  
  \node [draw,circle] (d) at (\xpos-0.5,\ypos+0.5){};
  \node [draw,circle] (e) at (\xpos+0.5,\ypos+0.5){};
  \node [draw,circle] (c) at (\xpos,\ypos){};
  \node [draw,circle] (a) at (\xpos-0.5,\ypos-0.5){};
  \node [draw,circle] (b) at (\xpos+0.5,\ypos-0.5){};

  \draw (a) -- (c) -- (b) -- (e) -- (d) -- (a)   (d) -- (c) (c) -- (e);
  
  \draw (a)--(q)--(b);
  
  \def\xpos{4}
  \def\ypos{2.3}
  
  \node [draw,circle] (dn) at (\xpos-0.5,\ypos+0.5){};
  \node [draw,circle] (en) at (\xpos+0.5,\ypos+0.5){};
  \node [draw,circle] (cn) at (\xpos,\ypos){};
  \node [draw,circle] (an) at (\xpos-0.5,\ypos-0.5){};
  \node [draw,circle] (bn) at (\xpos+0.5,\ypos-0.5){};

  \draw (an) -- (cn) -- (bn) -- (en) -- (dn) -- (an)   (dn) -- (cn) (cn) -- (en);
  
  \draw (an)--(t)--(bn);
  
  \end{tikzpicture}
	}
	\caption{}
	\label{fig:aux-clause-gadgets}
\end{figure}

Now, consider a linked $\varphi$-graph ${\cal L}$, with link ${\cal T}$. We first explain how to construct the edge gadgets that will replace the edges in ${\cal T}$.
Given an edge $e=c_ic_j\in E({\cal T})$, we denote the related \emph{edge gadget} by $\Xi_{ij}$ (observe Figure~\ref{fig:aux-clause-gadgets}.(b) to follow the construction). Consider the graph with vertex set $\{c_i,c_j,q_{ij}, r_{ij}, s_{ij}, t_{ij}\}$ and having as edges $\{c_iq_{ij}, c_ir_{ij},c_js_{ij},c_jt_{ij},r_{ij}s_{ij},q_{ij}r_{ij},s_{ij}t_{ij}\}$. To obtain $\Xi_e$  we add two copies of the auxiliary graph, one pendant at $q_{ij}$ and the other pendant at $t_{ij}$.


Finally, the \emph{variable gadget} is the graph $\Upsilon_i$ on vertices $A_i = \{a_i^1,\ldots, a_i^8\}$, $B_i = \{b_i^1,\ldots, b_i^8\}$ and $\{x_i^1,x_i^2,x^a_i,\bar{x}^1_i,\bar{x}^2_i\}$ depicted in Figure~\ref{fig:var-gadget}. The dotted edges denote how these graphs will be linked to the clause vertices.

\begin{figure}[!ht]
	\centering
	
	\begin{tikzpicture}[scale=1]
  \pgfsetlinewidth{1pt}
  \tikzset{vertex/.style={circle, minimum size=0.7cm, fill=black!20, draw, inner sep=1pt}}
  \node [vertex] (v) at (0,0){$v_i$};
  \node [vertex] (b1) at (-1,1){$b_i^1$};
  \node [vertex] (b2) at (-2,2){$b_i^2$};
  \node [vertex] (b3) at (-1,2){$b_i^3$};
  \node [vertex] (b4) at (0,2){$b_i^4$};
  \node [vertex] (b5) at (1,2){$b_i^5$};
  \node [vertex] (b6) at (2,2){$b_i^6$};
  \node [vertex] (b7) at (-0.5,3){$b_i^7$};
  \node [vertex] (b8) at (0.5,3){$b_i^8$};
  \node [vertex] (xb1) at (-1,4){$\bar{x}_i^1$};
  \node [vertex] (xb2) at (1,4){$\bar{x}_i^2$};

  \node [vertex] (a1) at (-1,-1){$a_i^1$};
  \node [vertex] (a2) at (-2,-2){$a_i^2$};
  \node [vertex] (a3) at (-1,-2){$a_i^3$};
  \node [vertex] (a4) at (0,-2){$a_i^4$};
  \node [vertex] (a5) at (1,-2){$a_i^5$};
  \node [vertex] (a6) at (2,-2){$a_i^6$};
  \node [vertex] (a7) at (-0.5,-3){$a_i^7$};
  \node [vertex] (a8) at (0.5,-3){$a_i^8$};
  \node [vertex] (x1) at (-1,-4){$x_i^1$};
  \node [vertex] (x2) at (1,-4){$x_i^2$};
  
  \draw (v)-- (b1) -- (b2) -- (b3) -- (b4) -- (b5) -- (b6) -- (v)  (b4)-- (b1) -- (b3) (b4) -- (v) -- (b5) (b3)--(b7)--(b4)--(b8)--(b5) (b2)--(xb1)--(b7) (b8)--(xb2)--(b6) (xb1)--(xb2);
  
  \draw (v)-- (a1) -- (a2) -- (a3) -- (a4) -- (a5) -- (a6) -- (v)  (a4)-- (a1) -- (a3) (a4) -- (v) -- (a5) (a3)--(a7)--(a4)--(a8)--(a5) (a2)--(x1)--(a7) (a8)--(x2)--(a6) (x1)--(x2);

  \def\xpos{-1}
  \def\ypos{5.3}
  
  \node [draw,circle] (d) at (\xpos-0.5,\ypos+0.5){};
  \node [draw,circle] (e) at (\xpos+0.5,\ypos+0.5){};
  \node [draw,circle] (c) at (\xpos,\ypos){};
  \node [draw,circle] (a) at (\xpos-0.5,\ypos-0.5){};
  \node [draw,circle] (b) at (\xpos+0.5,\ypos-0.5){};

  \draw (a) -- (c) -- (b) -- (e) -- (d) -- (a)   (d) -- (c) (c) -- (e);
  
  \draw (a)--(xb1)--(b);
  
  \node [vertex,dashed] (cl) at (1,5.3){$c_{\ell}$};
  \path [-] (xb1) edge [dashed] node {} (cl);
  \path [-] (xb2) edge [dashed] node {} (cl);
  
  \node [vertex] (xa) at (0,-5){$x_i^a$};
  
  \draw (x1)--(xa)--(x2);
  
  \node [vertex,dashed] (cj) at (-1.5,-5.3){$c_j$};
  \path [-] (x1) edge [dashed] node {} (cj);
  \path [-] (xa) edge [dashed] node {} (cj);
  
  \node [vertex,dashed] (ck) at (1.5,-5.3){$c_k$};
  \path [-] (x2) edge [dashed] node {} (ck);
  \path [-] (xa) edge [dashed] node {} (ck);
  
  \end{tikzpicture}
	\caption{Variable gadget $\Upsilon_i$.}
	\label{fig:var-gadget}
\end{figure}


We are now ready to prove Theorem~\ref{thm:np-c}, which we recall:

\thm{\ref{thm:np-c}}{\textsc{Hull Number in Cycle Convexity} is \NP-Complete, even if the input graph $G$ is a simple planar graph.}
\begin{proof}
	First, observe that \textsc{Hull Number in Cycle Convexity} is in \NP, since for a given subset $S\subseteq V(G)$ one can compute $\hull(S)$ in polynomial time and decide whether $\hull(S) = V(G)$.

	To prove the \NP-hardness of the problem, we reduce the \textsc{Linked Planar exactly 3-bounded 3-\SAT} to \textsc{Hull Number in Cycle Convexity}. Let $\varphi$ be an instance of the considered 3-\SAT problem, and let ${\cal L}$ be a linked $\varphi$-graph with link ${\cal T}$. We can consider ${\cal T}$ to be a tree as otherwise we can simply remove some edges. 
	We construct a graph ${\cal L}^*$ such that ${\cal L}^*$ is planar and prove that $\varphi$ is satisfiable if, and only if, $\hncc({\cal L}^*)\leq 4n + 4m-4$. This graph is obtained from ${\cal L}$ as follows: 

\begin{itemize}
\item For each variable $x_i$ of $\varphi$, let $c_j,c_k$  be the clauses in which $x_i$ appears posivitively and $c_\ell$ be the clause in which $x_i$ appears negatively. Replace vertices $x_i,\bar{x}_i$ by $\Upsilon_i$, and add edges $\{c_jx^1_i,c_jx_i^a,c_kx_i^2,c_kx_i^a,c_\ell \bar{x}_i^1, c_\ell\bar{x}_i^2\}$;
\item Replace each edge $c_ic_j$ in $E({\cal T})$ by $\Xi_{ij}$. 
\end{itemize}

	One can verify that ${\cal L}^*$ is planar by construction. Moreover, observe that ${\cal L}^*$ has less than $27n+14m$ vertices, since each variable gadget has~27 vertices, each edge gadget has~14 new vertices, and  since there are $m-1$ edges in ${\cal T}$. Before we can present our proof, we still need some facts concerning hull sets of ${\cal L}^*$, presented below.

\begin{claim}\label{claim:npcomplete}\ 

\begin{enumerate}
\item If $H\subseteq {\cal L}^*$ is an induced subgraph isomorphic to $\Lambda$ and is pendant at $v$, then $|V(H)\cap S|\ge 2$, for every hull set $S$ of ${\cal L}^*$;
\item For every variable $x_i$ of $\varphi$, and every hull set $S$ of ${\cal L}^*$, we have $|(A_i\cup B_i\cup \{v_i\})\cap S|\ge 2$;
\item For every edge $uv\in {\cal L}^*[A_i\cup\{v_i\}]$, we have $A_i\cup\{v_i,x_i^1,x^2_i,x^a_i,c_j,c_k\}\subseteq \hull(\{u,v\})$. The similar statement holds for $B_i$ and the corresponding vertices; and
\item Let $c_ic_j\in E({\cal T})$, and let $S\subseteq V({\cal L}^*)$ not containing $r_{ij}$. If $c_i\notin \icc^{k-1}(S)$, then $r_{ij}\notin \icc^{k}(S)$. The analogous holds for $c_j$ and $s_{ij}$.
\end{enumerate}
\end{claim}
\begin{claimproof}
Denote $V(H)$ by $C$, and let $v$ be the vertex of ${\cal L}^*$ on which $H$ is pendant (recall that $v\notin C$). Observe that $C$ is a co-convex set and, by Proposition~\ref{prop:co-convex}, it follows that $S\cap C\neq\emptyset$. One can also verify that, because $\{v,a,b\}$ separates $\{c,d,e\}$ from the rest of the graph, and since $ab\notin E(H)$, we get that $C\setminus \{u\}$ is also a co-convex set, for every $u\in C$. This implies Item (1).

Item (2) follows similarly. Let $C = (A_i\cup B_i\cup \{v_i\})$. Again, note that $C$ is co-convex and that $C\setminus \{u\}$ is also co-convex, for every $u\in C$.

To verify Item (3), first observe that $A_i\cup\{v_i\}$ induces a graph formed only by triangles and can be entirely contaminated by any of its edges. The item follows because $\{x_i^1,x_i^2,x_i^a\}$ is clearly contained in $\hull(A_i\cup\{v_i\})$.

Finally, let $S$ be as in Item (4), and suppose by contradiction that $c_i\notin \icc^{k-1}(S)$ and $r_{ij}\in \icc^{k}(S)$. This means that there exists a component of ${\cal L}^*[\icc^{k-1}(S)]$ containing $q_{ij}$ and $s_{ij}$. This is a contradiction since neither $c_i$ nor $r_{ij}$ is in $\icc^{k-1}(S)$, and $\{c_i,r_{ij}\}$ separates $q_{ij}$ from $s_{ij}$.
\end{claimproof}\vspace{0.2cm}

	Let us now prove that $\varphi$ is satisfiable if, and only if, $\hncc({\cal L}^*)\leq 4n+4m-4$.

Suppose first that $\varphi$ is satisfiable. Let us construct a hull set $S\subseteq V({\cal L}^*)$ such that $|S| = 4n+4m-4$. For every pendant auxiliary graph in ${\cal L}^*$, add any two adjacent vertices of such auxiliary graph to $S$, say vertices $d$ and $e$. Note that there are $n+2m-2$ such pendant auxiliary graphs, one for each variable and two for each edge of the tree ${\cal T}$. Finally, consider a truth assignment to the variables $x_1,\ldots,x_n$ satisfying $\varphi$. In case $x_i$ is true, add to $S$ the vertices $v_i$ and $a_i^1$. Otherwise, add to $S$ the vertices $v_i$ and $b_i^1$. We then get $|S| = 2\cdot(n+2m-2)+2n = 4n+4m-4$. Observe that this is the smallest size of a hull set by Claim~\ref{claim:npcomplete}, Items (1) and (2). It remains to show that $S$ is indeed a hull set.


	It is easy to verify that: (*) $H\cup \{v\}\subseteq\hull(S)$ for every auxiliary graph $H$ pendant at a vertex $v$. Also, by Claim~\ref{claim:npcomplete}, Item (3), we get that: if $x_i$ is true, then $A_i\cup\{v_i,x_i^1,x^2_i,x^a_i,c_j,c_k\}\subseteq \hull(\{v_i,a^1_i\})\subseteq \hull(S)$; while if $x_i$ is false, then $B_i\cup\{v_i,\bar{x}^1_i, \bar{x}^2_i,c_\ell\}\subseteq \hull(\{v_i,b_i^1\})\subseteq \hull(S)$. Because every gadget contains a truth literal, this means that ${\cal C}\subseteq \hull(S)$. This and (*) imply that $\hull(S)$ contains every edge gadget, which means that: (**) ${\cal C}$ is contained in the same connected component of $\hull(S)$. Now it remains to prove that the ``untruth'' side of the variable gadgets are also in $\hull(S)$. So consider any variable $x_i$, and let $c_j,c_k$ be the clauses that contain $x_i$ positively, and $c_\ell$ be the clause containing the negation of $x_i$. First, suppose that $x_i$ is true. By (*) and (**),  we get $\{\bar{x}^1_i,c_\ell\}\subseteq\hull(S)$, and hence $\bar{x}^2_i\in\hull(S)$. But in this case, by (**) and since $\{v_i,a_i^6,x^2_i,c_k\}\subseteq \hull(S)$, we get that $b_i^6\in\hull(S)$. It follows from Claim~\ref{claim:npcomplete} that $B_i\subseteq\hull(\{v_i,b^6_i\})\subseteq \hull(S)$. When $x_i$ is false, by (**), we get that $\hull(S)$ contains $x_i^a$ and consequently $x_i^2$ and $a_i^6$; the rest follows similarly.

	Now, let $S$ be a hull set of ${\cal L}^*$ such that $|S|\leq 4n+4m-4$. As before, we know that equality must occur by Claim~\ref{claim:npcomplete}, Items (1) and (2). Let $x_i$ be true if, and only if, $A_i\cap S \neq \emptyset$. One can verify that Claim~\ref{claim:npcomplete}, Item (2), and the size of $S$ tell us that at most one between $A_i$ and $B_i$ has non-empty intersection with $S$; hence, the assignment is well-defined. We prove that this assignment satisfies $\varphi$.   For this, consider $c_i$ to be the clause $(\ell_{i_1}\vee \ell_{i_2}\vee \ell_{i_3})$, where each $\ell_{i_j}$ denotes either the literal $x_{i_j}$ or $\bar{x}_{i_j}$, and suppose by contradiction that $c_i$ is not satisfied. We prove that in this case $c_i\notin\hull(S)$, thus getting a contradiction.

Let $k\in\Nat^*$ be such that $c_i\in \icc^k(S)\setminus \icc^{k-1}(S)$, and denote by $S^{k-1}$, $S^k$ the sets $\icc^{k-1}(S)$, $\icc^{k}(S)$, respectively. Let $H$ be the component of ${\cal L}^*[S^{k-1}]$ containing two neighbors of $c_i$; call them $u$ and $v$. Let $c_ic_j$ be an edge of ${\cal T}$ and note that, by Claim~\ref{claim:npcomplete}, Item (1), and the cardinality of $S$,  we know that $r_{i,j}\notin S$. By Claim~\ref{claim:npcomplete}, Item (4), we get that $r_{i,j}\notin S^{k-1}$. This means that $u$ and $v$ must be inside the variable gadgets related to $c_i$. 

Now, suppose that $\ell_{i_1}=x_{i_1}$, i.e., the literal is a positive occurrence of the variable; also suppose, without loss of generality, that $c_i$ is adjacent to $x_{i_1}^1$. Because $c_i$ is not satisfied, we know that $A_{i_1}\cap S=\emptyset$. Also, observe that no vertex of $A_{i_1}$ can be contaminated before $x_{i_1}^1$ and $x^2_{i_1}$, and in turn these vertices cannot be contaminated before $x^a_{i_1}$. However, since $x^a_{i_1}$ have exactly two neighbors outside of the edge gadget, we get that it depends on $c_i$ to be contaminated. In short, the facts that $c_i\notin S^{k-1}$ and $A_{i_1}\cap S=\emptyset$ imply $\{x^a_{i_1},x^1_{i_1},x^2_{i_1}\}\cap S^{k-1}=\emptyset$. By applying a similar argument for the case $\ell_{i_1}=\bar{x}_{i_1}$, and since $i_1$ was arbitrarily chosen from $\{i_1,i_2,i_3\}$, we get that, in fact, $c_i$ has no contaminated neighbors inside the related variable gadgets, a contradiction.
%
%
	\end{proof}


\section{Polynomial-time algorithm for $P_4$-sparse graphs}\label{sec:poly}

In this section, we just refer to simple graphs.

One should notice that, by the equalities presented in Section~\ref{sec:preliminaries}, the hull number in the cycle convexity of chordal graphs and grids can be obtained in linear time.
In this section, we prove that one can compute $\hncc(G)$ in polynomial time when $G$ is a $P_4$-sparse graph.

It is well known that $P_4$-sparse graphs have a very nice decomposition, but before we present it, we need some new definitions. 

Given graphs $G_1$ and $G_2$, the \emph{union} of $G_1$ and $G_2$ is the graph $G_1\vee G_2 = (V(G_1)\cup V(G_2), E(G_1)\cup E(G_2))$. The \emph{join} of $G_1$ and $G_2$ is the graph obtained from $G_1\vee G_2$ by adding every possible edge having one endpoint in $V(G_1)$ and the other in $V(G_2)$, and it is denoted by $G_1\wedge G_2$.

A graph $G$ is a \emph{spider} if $V(G)$ can be partitioned into sets $K$, $S$ and $R$ such that:
\begin{itemize}
    \item $|K|=|S|\ge 2$ and $R$ may be empty;
    \item $K$ is a clique, $S$ is a stable set;
    \item each vertex $R$ is adjacent to each vertex of $K$ and to no vertex of $S$;
    \item there is a bijection $f:S\to K$ such that either $N(s) = \{f(s)\}$ for every $s\in S$ (in which case we say that $G$ is a \emph{thin spider}), or $N(s)=K\setminus \{f(s)\}$ for every $s\in S$ (in which case we say that $G$ is a \emph{fat spider}).
\end{itemize}

The subgraph $G[R]$ is called the \emph{head} of the spider.

\begin{theorem}[\cite{JO.95}]\label{theo:decomposition}
Let $G$ be a non-trivial $P_4$-sparse graph. Then, one of the following holds:
\begin{enumerate}
  \item $G$ is the union of two $P_4$-sparse graphs;
  \item $G$ is the join of two $P_4$-sparse graphs; or
  \item $G$ is a spider with partition sets $K,S,R$ such that $G[R]$ is a $P_4$-sparse graph.
\end{enumerate}
\end{theorem}

While the join and the disjoint union may easily be seen as graph operations, the spider is often also considered to be one, in which we take a $P_4$-sparse graph $H$ and we place $H$ as the head of some spider.

First, we investigate the union and the join of graphs.

\begin{lemma}
\label{lem:join_disjunion}
If $G_1$ and $G_2$ are graphs such that $|V(G_1)|\le |V(G_2)|$ and $G_2$ has $p$ components, then:
\[\hncc(G_1\vee G_2) = \hncc(G_1)+\hncc(G_2)\mbox{, and}\]
\[\hncc(G_1\wedge G_2) = \left\{
\begin{array}{ll}
p+1 & \mbox{, if $|V(G_1)| = 1$;}\\
2 & \mbox{, if $|V(G_1)|>1$ and $E(G_i)\neq \emptyset$ for some $i\in \{1,2\}$;}\\
3 & \mbox{, otherwise.}
\end{array}\right.\]
\end{lemma}
\begin{proof}
As the first equation is straightforward, we just prove the second one.

Let $G=G_1\wedge G_2$. Suppose first that $V(G_1) = \{u\}$. Then, observe that $u$ is a universal vertex (i.e. $N(u)=V(G)\setminus\{u\}$), and that the blocks of $G$ are the subgraphs of $G$ composed by the components of $G_2$ plus the vertex $u$. By Corollary~\ref{cor:hngep+1}, we get that $\hncc(G)\ge p+1$. On the other hand, because $u$ is universal in $G$, for each component $C$ of $G_2$ and every $v\in V(C)$, we have that $V(C)\subseteq \hull(\{u,v\})$. This means that if we pick one vertex of each component of $G_2$ together with vertex $u$, we obtain a hull set $S$ of $G$ with $p+1$ vertices. Thus, $\hncc(G)\le p+1$. 

Now, suppose that $1<|V(G_1)|\leq V(G_2)$ and $E(G_2)\neq \emptyset$. Again, since $G$ is simple and $|V(G)|\geq 2$, note that $\hncc(G)\geq 2$. On the other hand, let $x,y\in V(G_1)$ and $uv\in E(G_2)$. We claim that $\{u,v\}$ is a hull set of $G$. In fact, note that $V(G_1)\subseteq \icc(\{u,v\})$. Then, for each vertex $w\in V(G_2)\setminus \{u,v\}$, observe that there is a cycle in $G$ formed by the vertices $w,x,u,y$. Consequently, $V(G_2)\subseteq \iccp{2}(\{u,v\})$. The same argument can be applied if $E(G_1)\neq \emptyset$ instead.

Finally, suppose that $1<|V(G_1)|\leq V(G_2)$ and $E(G_1) = E(G_2) = \emptyset$. Consequently, $G$ is bipartite and, thus, triangle-free. Therefore, since $G$ is simple and $|V(G)|\geq 4$, note that $\hncc(G)\geq 3$. On the other hand, let $x,y\in V(G_1)$ and $u,v\in V(G_2)$. We claim that $S=\{x,y,u\}$ is a hull set of $G$. In fact, note that, for each vertex $w\in V(G_2)\setminus \{u\}$, there is a cycle in $G$ formed by the vertices $w,x,u,y$. Consequently, $V(G_2)\subseteq \icc(\{u\})$. Now, observe that for each vertex $w\in V(G_1)\setminus\{x,y\}$, there is a cycle in $G$ formed by the vertices $w,v,x,u$. Consequently, $V(G_1)\subseteq \iccp{2}(\{u\})$.
\end{proof}

The next lemma is the last needed tool.

\begin{lemma}
\label{lem:spider}
Let $G$ be a spider with partition sets $\{K,S,R\}$ such that $K$ is a clique and $S$ is a stable set. Then,
\[\hncc(G) = \left\{\begin{array}{ll}
2 & \mbox{, if $|K|\ge 3$ and $G$ is a fat spider,}\\
2+|S| & \mbox{, otherwise.}
\end{array}\right.\]
\end{lemma}
\begin{proof}
Once more, recall that $G$ is simple and $|V(G)|\geq 4$ (since $|S|=|K|\geq 2$) and therefore $\hncc(G)\geq 2$.

By definition, note that if $|K|\ge 3$ and $G$ is a fat spider, then every vertex of $G$ has at least two neighbors in $K$. Consequently, we claim that a subset $S = \{u,v\}$ such that $u,v\in K$ is a hull set of $G$. In fact, note that $K\subseteq \icc(S)$ and thus $V(G)\subseteq\iccp{2}(S)$. It follows that $\hncc(G)=2$.

Now observe that if $|K|=2$, then $G$ is both fat and thin. Thus, we may just assume that $G$ is thin. By definition, each vertex of $S$ has exactly one neighbor in $K$.
Since each vertex of $S$ has degree one, by Corollary~\ref{cor:removesingletoncoconvex}, we deduce that $\hncc(G)=\hncc(G-S)+|S|$. We just have to prove that $\hncc(G-S)=2$. If $R=\emptyset$, note that the graph $G-S = G[K\cup R]$ is a complete graph on $|K|\geq 2$ vertices and thus, as observed in Proposition~\ref{prop:complete}, $\hncc(G-S)=2$.
Otherwise, we can see $G-S$ as the join of two graphs $G_1=G[R]$ and $G_2=G[K]$. Depending on whether $|R|=1$, one may verify that $\hncc(G-S)=2$ by using Lemma~\ref{lem:join_disjunion}.
\end{proof}

The operations presented in Theorem~\ref{theo:decomposition} actually characterize the modular decomposition of $P_4$-sparse graphs, which can be obtained in linear time~\cite{MS94}. This decomposition is often used to obtain polynomial-time algorithms to compute graph parameters~\cite{JO.95}. Such decomposition is represented by a rooted tree in which the leaves are the vertices of $G$ and each non-leaf vertex $t$ represents a subgraph $G_t$ of $G$ obtained after the application of one the operations - join, disjoint union, spider - over the subgraphs $G_{t_1}, \ldots, G_{t_p}$ represented by the children $t_1,\ldots, t_p$ of $t$ in $T$. Thus, in a post-order traversal of this tree $T$, we actually build the graph $G$ from trivial graphs by applying one of the operations. This means that we can use such tree to compute $\hncc(G)$, whenever $G$ is a $P_4$-sparse graph.

\begin{theorem}
If $G$ is a $P_4$-sparse graph, then $\hncc(G)$ can be computed in linear time.
\end{theorem}
\begin{proof}
By using Lemmas~\ref{lem:join_disjunion} and~\ref{lem:spider}, we need only to traverse the decomposition tree of $G$ in post-order to compute $\hncc(G)$. Because the decomposition tree can be computed in linear time, the theorem holds.
\end{proof}

%
%
%
%
%
%

\section*{Dedicatory}
While this paper was under preparation, Darlan Gir\~ao was diagnosed with cancer. After a prolonged courageous and dignifying battle with his condition, Darlan died before we could finish this work together. Darlan has been a very good friend and colleague, who we will miss. This paper is dedicated to his memory.

\bibliographystyle{elsarticle-num}
\bibliography{CycleConvexity}

\end{document}